\documentclass[11pt,pdftex]{amsart}
\usepackage{amssymb}
\usepackage{amscd}
\usepackage{comment}
\usepackage{tikz}
	\usetikzlibrary {arrows.meta}
	\usetikzlibrary {bending}
	\usetikzlibrary{calc}
\usepackage{mathrsfs}
\usepackage{braket}
\usepackage{graphicx}
\newtheorem{dfn}[equation]{Definition}
\newtheorem{prop}[equation]{Proposition}
\newtheorem{thm}[equation]{Theorem}
\newtheorem{cor}[equation]{Corollary}
\newtheorem{lem}[equation]{Lemma}
\newtheorem{rmk}[equation]{Remark}
\newtheorem{asmp}[equation]{Assumption}

\numberwithin{equation}{section}
\newcommand{\Z}{\mathbb{Z}}
\newcommand{\Q}{\mathbb{Q}}

\newcommand{\C}{\mathbb{C}}
\newcommand{\Zp}{\mathbb{Z}_p}

\newcommand{\Gal}{{\mathrm{Gal}}}
\newcommand{\Hom}[1]{\mathrm{Hom}_{#1}}

\newcommand{\Fitt}{{\mathrm{Fitt}}}
\newcommand{\Cl}[1]{\mathrm{Cl}_{#1}}

\newcommand{\Pic}[1]{\mathrm{Pic}({#1})}
\newcommand{\Picz}[1]{\mathrm{Pic}^0({#1})}
\newcommand{\ord}[1]{\mathrm{ord}_{#1}}
\newcommand{\coker}{{\mathrm{coker}}}
\newcommand{\rank}{{\mathrm{rank}}}
\newcommand{\pe}{{\mathfrak{p}}}
\newcommand{\ROI}[1]{{\mathcal{O}_{#1}}}
\newcommand{\UG}[1]{{\mathcal{O}_{#1}^\times}}
\newcommand{\CG}[1]{{\widehat{#1}}}
\newcommand{\hks}{h_{K,S}}
\newcommand{\Rks}{R_{K,S}}

\newcommand{\BF}[1]{\mathrm{BF}({#1})}
\newcommand{\BFtors}[1]{{\mathrm{BF}({#1})_{\mathrm{tors}}}}
\newcommand{\BFp}[1]{\mathrm{BF}_p({#1})}

\newcommand{\BFopsi}[1]{\mathrm{BF}_{O_\psi}({#1})}

\title[BOWEN--FRANKS GROUPS AND MINUS CLASS GROUPS]{BOWEN--FRANKS GROUPS AND MINUS CLASS GROUPS OF IMAGINARY ABELIAN FIELDS}
\author{Takuya Murakami}
\address{Osaka Metropolitan University, Department of Mathematics,
3-3-138 Sugimoto, Sumiyoshi-ku, Osaka, 558-8585, Japan}
\address{Osaka Prefectural Yamada Senior High School,
3-28-1 Yamadahigashi, Suita-shi, Osaka, 565-0821, Japan}
\email{takuyamurakami.jp@gmail.com}
\date{}
\begin{document}
\begin{abstract}
Lei, M\"uller, and Valli$\mathrm{\grave{e}}$res constructed a directed graph for which the torsion part of the corresponding Bowen--Franks group is related to the minus part of the class group of a cyclotomic number field with prime conductor. 
We extend this result to a general imaginary abelian field.
\end{abstract}
\maketitle
\tableofcontents
\section{Introduction}\label{_Intro}
\subsection{Main results}\label{__Intro1}
Let $p$ be an odd rational prime.
Let $\zeta_p$ be a primitive $p$-th root of unity.
In \cite[Theorem A]{LMV},
Lei, M\"uller, and Valli$\mathrm{\grave{e}}$res
constructed a directed graph $Y$ related to 
the Stickelberger element $\theta_p$
for the Galois group $\Gal(\Q(\zeta_p)/\Q)$
(see \S\ref{SE} for the precise definition)
and proved the following formula:
\begin{equation}\label{MTH_LMV}
\#\BFtors{Y}=p^{\frac{p-1}{2}}\cdot h_{\Q(\zeta_p)}^-,
\end{equation}
where $\BF{Y}$ and $h_{\Q(\zeta_p)}^-$
denote the Bowen--Franks group of $Y$ and
the minus class number of the cyclotomic field $\Q(\zeta_p)$
with conductor $p$.

In this paper,
we extend the result \eqref{MTH_LMV} to
general imaginary abelian fields.
Let $m$ be an integer such that $m>2$ and
assume that $4\mid m$ if $2\mid m$.
Let $S$ be the finite set of places of $\Q$
consisting of the primes dividing $m$
and the infinite place $\infty$.
\begin{thm}\label{intro_MTH}
Let $K$ be an imaginary abelian field with conductor $m$.
Let $t$ be the number of finite places
of the maximal real subfield of $K$ which lie above the primes dividing $m$
and are inert in $K$.
We assume the ``no trivial zeros'' condition (Assumption~\ref{NTZ})
for $\Gal(K/\Q)$,
that is,
we assume that $L_S(0,\chi)\ne0$
for all odd characters $\chi$ of $\Gal(K/\Q)$,
where $L_S(s,\chi)$ denotes the $S$-truncated Dirichlet $L$-function.
Then there exists a digraph $Y$ such that
$$
\#\BFtors{Y}=m^{\frac{[K:\Q]}{2}+1}\hspace{1pt}2^{t+1}
\frac{[\Q(\zeta_m):K]}{w_K\cdot\mathcal{Q}_K}
h_K^-,
$$
where 
$h_K^-$ and $w_K$ denote 
the minus class number of $K$ and
the number of roots of unity in $K$
respectively, and
$\mathcal{Q}_K:=[\UG{K}:\mu_K\UG{K^+}]\in\{1,2\}$ is the Hasse unit index.
The notation used here is defined in \S\ref{__Intro3}.
\end{thm}
We note that Assumption \ref{NTZ}
is automatically satisfied, for example,
when $K$ is an imaginary quadratic field or
a cyclotomic field $\Q(\zeta_{p^n})$
for a rational prime $p$ and
a positive integer $n$ with $n>1$ if $p=2$
(see Remark \ref{rmk_NTZ_holds}).

Let $p$ be an odd rational prime.
If $K=\Q(\zeta_p)$, then
we have $m=p$, $w_K=2p$, $\mathcal{Q}_K=1$, and $t=0$.
Hence,
the right hand side of Theorem \ref{intro_MTH}
is equal to $p^{\frac{p-1}{2}}h_K^-$,
which is nothing but the result (\ref{MTH_LMV})
of Lei, M$\mathrm{\ddot{u}}$ller, and Valli$\mathrm{\grave{e}}$res.

To explain the difference between \eqref{MTH_LMV} and Theorem \ref{intro_MTH},
we recall some classical results.
Let $K$ be an imaginary abelian field.
The formula for
the minus class number $h_K^-$ of $K$
is given by
$$
h_K^-=\mathcal{Q}_Kw_K\prod_{\chi:\mathrm{odd}}\left(-\frac{1}{2}B_{1,\chi}\right),
$$
where 
the product is taken over all $primitive$ characters induced by
odd characters of $\Gal(K/\Q)$
and
$B_{n,\chi}$ are
the generalized Bernoulli numbers.
It is known that
$$
-B_{1,\chi}=L(0,\chi),
$$
where $L(s,\chi)$ is the Dirichlet $L$-function
for a character $\chi$ of $\Gal(K/\Q)$.
Via these classical formulas,
the minus class number
is related to the values of Dirichlet $L$-functions at $s=0$.
If $K$ is the cyclotomic field $\Q(\zeta_m)$ with
conductor $m$, then
for any $primitive$ character $\chi$ with conductor $m$,
it is known that
$$
\chi^{-1}(\theta_m)=L(0,\chi),
$$
where
$\theta_m$ denotes the Stickelberger element for $\Gal(\Q(\zeta_m)/\Q)$.

In \cite{LMV},
for the case where $m=p$,
Lei, M$\mathrm{\ddot{u}}$ller, and Valli$\mathrm{\grave{e}}$res
proved \eqref{MTH_LMV}
by constructing a directed graph $Y$ related to $\theta_p$.
However, even if $K$ is a general cyclotomic field $\Q(\zeta_m)$,
the characters of $\Gal(\Q(\zeta_m)/\Q)$ are not necessarily primitive,
and thus the usual relative class number formula described above
cannot be applied directly.
Therefore, we assume the ``no trivial zeros'' condition
(see Assumption \ref{NTZ})
and apply the minus $S$-class number formula
(see Theorem \ref{mSCNF2}).
The number $t$ appearing in Theorem \ref{intro_MTH}
arises when connecting the minus $S$-class number with
the ordinary minus class number.
When $K$ is an imaginary abelian field with conductor $m$,
we can regard the character group of $\Gal(K/\Q)$ as
a subgroup of the character group of $\Gal(\Q(\zeta_m)/\Q)$.
Thus, we can apply essentially the same argument
as in the cyclotomic case.
The only difference from the cyclotomic case
is the factor $[\Q(\zeta_m):K]$,
which arises when projecting the norm element over $\Q(\zeta_m)$ to $K$.

Fix an odd rational prime $p$.
Let $\psi$ be an odd character of $G:=\Gal(K/\Q)$.
The second main result of this paper gives
a formula relating the orders of
the $\psi$-components of the $p$-adic completion of the Bowen--Franks group
$\BF{Y}\otimes_{\Z[G]}\Z_p[\psi(G)]$
and the ideal class group $\Cl{K}\otimes_{\Z[G]}\Z_p[\psi(G)]$.
This result extends \cite[Theorem B]{LMV}.
\begin{thm}\label{intro_MTH2}
Let $K/\Q$ be an imaginary abelian extension with conductor $m$
and Galois group $G$.
Let $\psi$ be an odd character of $G$.
We assume that for every rational prime $\ell\mid m$, the value at $\ell$
of the primitive character associated with $\psi$
is not equal to one
(Assumption~\ref{NTZpsi}).
Fix an odd rational prime $p$.
If $\Q(\zeta_p)\subset K$,
assume that
the restriction of $\psi$ to
the unique maximal subgroup of $G$ of order prime to $p$
is not the Teichm\"uller character
(Assumption~\ref{psineomega}).
We denote by $\varpi$ a uniformizer of
the discrete valuation ring $\Z_p[\psi(G)]$.
Let $Y$ be the digraph constructed
in the proof of Theorem \ref{intro_MTH}.
Then, we have
$$\#(\BF{Y}\otimes_{\Z[G]}\Z_p[\psi(G)])
=|m\varpi^{c_\psi}|_p^{-\rank_{\Z_p}(\Z_p[\psi(G)])}
\#(\Cl{K}\otimes_{\Z[G]}\Z_p[\psi(G)]),$$
where $c_\psi$ is the constant defined in Theorem \ref{__MTH2}
and $|\cdot|_p$ denotes the usual $p$-adic absolute value.
In particular, if $p\nmid\# G$, then we have
$$\#(\BF{Y}\otimes_{\Z[G]}\Z_p[\psi(G)])
=p^{\delta}
\#(\Cl{K}\otimes_{\Z[G]}\Z_p[\psi(G)]),$$
where $\delta=0$ if $p\nmid m$, and $\delta=\rank_{\Z_p}(\Z_p[\psi(G)])$ if $p\mid m$.
\end{thm}

Let $\ell$ be an odd rational prime.
We consider the setting of \cite[Theorem~B]{LMV},
where $K=\Q(\mu_\ell)$ and $p\nmid\# G=\ell-1$.
Note that the roles of $p$ and $\ell$ are interchanged.
If $\ell\ne p$, then $\delta=0$.
On the other hand,
if $\ell=p$, then $\Z_p[\psi(G)]=\Z_p$, so $\delta=1$.
Hence,
the latter assertion of Theorem \ref{intro_MTH2}
recovers \cite[Theorem B]{LMV},
except for the case where $\ell=p$ and $\psi$ is the Teichm\"uller character.
\subsection{Background}\label{__Intro2}
In undirected graph theory,
the number of spanning trees $\kappa(X)$ of a finite connected multigraph $X$
is studied as an analogue of the class number
in number theory.
By Kirchhoff's matrix tree theorem,
this is known to be
the cardinality of the Picard group $\Picz{X}$
of degree zero on $X$.
In directed graph theory,
the group corresponding to the Picard group $\Pic{X}$
is the Bowen--Franks group $\BF{X}$.
The zeta function for directed graphs
that corresponds to the Ihara zeta function
for multigraphs
is the Artin--Mazur zeta function.
The Artin--Mazur zeta function was first introduced in \cite{AM65}.
Later, Bowen and Lanford showed in \cite{BL70} that 
the inverse of the Artin--Mazur zeta function is
the reverse characteristic polynomial of
the associated matrix.
For that reason, the zeta function is sometimes called
the Bowen--Lanford zeta function.
Moreover, Kotani and Sunada
reinterpreted this result
in the context of graph theory
in \cite{KS00}.
These results allow us to define 
the inverse of the Bowen--Lanford zeta function
of a directed graph $X$
as the reverse characteristic polynomial of
the adjacency matrix of $X$.
Following \cite{LMV},
we denote it by $g_X(u)$.
The equivariant Ihara zeta function was introduced in \cite{GV25}
for multigraphs.
Following this, the inverse of the equivariant Artin--Mazur zeta function
$\gamma_{Y/X}(u)$
was defined in \cite{LMV} for a Galois cover $Y/X$ of directed graphs.
The following formula, proved in \cite[Theorem 3.13]{LMV},
is an analogue of the decomposition of
the Dedekind zeta function into
Dirichlet $L$-functions in number theory:
$$
g_Y(u)=\prod_{\chi}\chi(\gamma_{Y/X}(u)),
$$
where the product is taken over
all characters of the group of deck transformations $\Gal(Y/X)$
Thanks to this decomposition,
we can consider a value corresponding  to the minus class number
in graph theory as well.

In the Iwasawa theory of graphs,
the graph analogue of Iwasawa's class number formula
in number theory (ICNF, for short)
has been studied
in the setting of $\Z_p^d$-towers of graphs.
In undirected graph theory,
the main object of study is the $p$-adic valuation
$\ord{p}(\kappa(X_n))$
of the cardinality $\kappa(X_n)$ of $\Picz{X_n}$,
where $X_n$ is the $n$-th layer of the $\Z_p^d$-tower.
Two types of coverings,
unramified and branched (or ramified),
have been defined,
and the graph analogue of the ICNF
has been proved in both settings
(see \cite[\S1.1]{K26} for details).
Furthermore, in \cite{LV23},
for a rational prime $\ell$ distinct from $p$,
the ICNF for $\ord{\ell}(\kappa(X_n))$ has been proven
for $\Z_p$-towers of unramified coverings.
On the other hand,
in directed graph theory,
only unramified coverings have been defined so far.
In \cite{LM26},
the ICNF for $\ord{\ell}(\#\BFtors{X_n})$
with $\ell\ne p$
was proven 
up to $O(p^{(d-1)n})$
for $\Z_p^d$-towers of (unramified) coverings.
Although the present paper does not focus on the ICNF itself,
it can be viewed as part of the Iwasawa theory of directed graphs.
\subsection{Notation}\label{__Intro3}
For the reader's convenience,
we review some general notation
that will be used throughout this paper.
For a number field $K$,
the ring of integers of $K$ is denoted by $\ROI{K}$.
For a finite set $S$ of places of $\Q$,
the set of places of $K$ lying above $S$ is denoted by $S_K$,
and the ring of $S$-integers of $K$ is denoted by $\ROI{K,S}$
instead of $\ROI{K,S_K}$.
We denote by $\UG{K}$ and $\UG{K,S}$
the unit group and the $S$-unit group of $K$
respectively.
Similarly,
we denote by $\Cl{K}$ and $\Cl{K,S}$
the class group and the $S$-class group of $K$,
whose cardinalities are denoted by $h_K$ and $h_{K,S}$
respectively.
The sets of finite and infinite places in $S_K$
are denoted by $S_{K,\mathrm{fin}}$ and $S_{K,\infty}$
respectively.
We denote by $\mu_K$ the group of roots of unity in $K$,
whose cardinality is denoted by $w_K$.
For an imaginary abelian field $K$,
the maximal real subfield is denoted by $K^+$.
For a positive integer $m$,
we denote by $\zeta_m$ a primitive $m$-th root of unity.
The cyclotomic field with conductor $m$ is denoted by $\Q(\zeta_m)$,
and its Galois group over $\Q$ is denoted by $G_m$.
For an integer $i$ relatively prime to $m$,
we denote by $\sigma_i$ the element of $G_m$
satisfying $\sigma_i(\zeta_m)=\zeta_m^i$.
For a finite abelian group $G$,
the norm element of $G$
is denoted by $N_G:=\sum_{\sigma\in G}\sigma \in\Z[G]$.
The character group of $G$ is denoted by
$\CG{G}:=\Hom{\Z}(G,\C^\times)$.
We denote by $\overline{\Q}$ the algebraic closure of $\Q$ in $\C$.
For each rational prime $p$,
we fix an embedding $\overline{\Q}\hookrightarrow\overline{\Q}_p$.
When working $p$-adically, we regard each $\psi\in\CG{G}$
as a $\overline{\Q}_p$-valued character via this embedding.
\subsection{Acknowledgments}\label{__Intro4}
The author would like to deeply thank
Professor Takamichi Sano
for his many valuable suggestions and helpful advice.
He would also like to thank
Kensho Okano
for many stimulating discussions.
\section{Class number formulas and Stickelberger elements}\label{_CNFandSE}
In this section, we first review the minus $S$-class number formula involving the special values of $S$-truncated Dirichlet $L$-functions at zero. 
Next, we discuss the ``no trivial zeros'' condition.
Under this condition, the minus $S$-class number formula can be described
in terms of the $S$-unit index and the values of $S$-truncated $L$-functions at zero.
Finally, we review some basic properties of Stickelberger elements.
\subsection{The minus $S$-class number formula}\label{__mSCNF}
Let $K$ be an imaginary abelian extension of $\Q$ with Galois group $G$.
The ring of integers of $K$ is denoted by $\ROI{K}$.
For a finite set $S$ of places of $\Q$ containing the infinite place $\infty$, the set of places of $K$ lying above $S$ is denoted by $S_K$, and the ring of $S$-integers of $K$ is denoted by $\ROI{K,S}$ instead of $\ROI{K,S_K}$.
The $S$-truncated Dedekind zeta function of $K$ is defined by 
$$
\zeta_{K,S}(s):=\prod_{\pe\not\in S_K}\frac{1}{1-N(\pe)^{-s}},
$$
where the product is taken over all finite places of $K$ not contained in $S_K$, and $N(\pe)$ denotes the cardinality of the residue field $\ROI{K}/\pe$ of $\pe$.
It is known that
\begin{equation}\label{r}
r:=\ord{s=0}\zeta_{K,S}(s)=\rank_\Z\UG{K,S}=\# S_K-1.
\end{equation}
The $S$-class number formula is stated as follows:
\begin{equation}\label{SCNF}
\zeta_{K,S}^*(0):=\lim_{s\to0}\frac{1}{s^r}\zeta_{K,S}(s)
=-\frac{\hks\cdot\Rks}{w_K},
\end{equation}
where 
$\hks, \Rks$, and $w_K$ denote the $S$-class number of $K$ (i.e.\ the class number of $\ROI{K,S}$),
the $S$-regulator of $K$, and 
the number of roots of unity in $K$
respectively.
For a character $\chi\in\CG{G}:=\Hom{\Z}(G,\C^\times)$, 
the $S$-truncated Dirichlet $L$-function is defined by
$$
L_S(s,\chi):=\prod_{p\not\in S}\frac{1}{1-\chi(p)p^{-s}}
\hspace{30pt}(\mathrm{Re}(s)>1)
$$
where the product is taken over all rational primes not contained in $S$.
The function $L_S(s,\chi)$ continues meromorphically to the whole complex plane, and its leading coefficient at zero is denoted by $L^*_S(0,\chi)$.
There exists a decomposition
\begin{equation}\label{decomp}
\zeta_{K,S}(s) = \prod_{\chi\in\CG{G}}L_S(s,\chi)
\end{equation}
and hence we have 
$$
\zeta_{K,S}^*(0) = \prod_{\chi\in\CG{G}}L_S^*(0,\chi).
$$

Let $K^+$ be the maximal real subfield of $K$.
The minus $S$-class number is defined by
$$
\hks^-:=\frac{\hks}{h_{K^+,S}}.
$$
(One can show that this is an integer.)
By (\ref{SCNF}) and (\ref{decomp}),
we have
\begin{equation}\label{mSCNF}
\prod_{\substack{\chi\in\CG{G}\\\chi:\mathrm{odd}}}L_S^*(0,\chi)
=\frac{\zeta_{K,S}^*(0)}{\zeta_{K^+,S}^*(0)}
=\frac{2\Rks}{w_K\cdot R_{K^+,S}}\hks^-.
\end{equation}
\subsection{The ``no trivial zeros'' condition}\label{__NTZ}
Let $K$ be an imaginary abelian extension of $\Q$ with Galois group $G$.
Let $S$ be a finite set of places of $\Q$ containing the infinite place $\infty$.
Throughout the remainder of $\S$\ref{_CNFandSE}
and in $\S$\ref{__MTH},
we impose the following assumption.
In $\S$\ref{__character components},
we replace it with the corresponding condition
for a fixed odd character $\psi$
(see Assumption \ref{NTZpsi}).
\begin{asmp}[``no trivial zeros'']\label{NTZ}
For all odd characters $\chi\in\CG{G}$, we have $L_S(0,\chi)\ne0$.
\end{asmp}

\begin{rmk}\label{rmk_NTZ}
Note that Assumption~\ref{NTZ} is equivalent to the following condition:
for any odd character $\chi\in\CG{G}$, 
we have $\chi^{\mathrm{prim}}(p)\ne 1$ for any rational prime $p\in S$,
where $\chi^{\mathrm{prim}}$ denotes the primitive character associated with $\chi$.
Under this assumption, we have
$$
\sum_{\chi\in\CG{G}}\ord{s=0}L_S(s,\chi)
=\sum_{\substack{\chi\in\CG{G}\\\chi:\mathrm{even}}}\ord{s=0}L_S(s,\chi)
$$
since $\ord{s=0}L_S(s,\chi)=0$ for any odd character $\chi\in\CG{G}$.
Therefore, combining this with \eqref{r} and \eqref{decomp}, we have $\# S_K=\# S_{K^+}$.
Hence,
$\# S_{K,\mathrm{fin}}=\# S_{K^+,\mathrm{fin}}$,
and thus no finite place in $S_{K^+}$ splits in $K$.
\end{rmk}
\begin{rmk}\label{rmk_NTZ_holds}
We present cases where Assumption \ref{NTZ} is satisfied.
\begin{enumerate}
\item
The case where $K=\Q(\zeta_{p^n})$ and $S=\{p,\infty\}$
for a rational prime $p$ and
a positive integer $n$
with $n>1$ if $p=2$.
Indeed, since every odd character $\chi\in \CG{G}$ has conductor $p^m$
for some $m\ (0< m\leq n)$,
we have $\chi^{\mathrm{prim}}(p)=0\ne1$. 
\item
The case where $K$ is an imaginary quadratic field
with conductor $m$
and $S=\{p\mid p\text{ divides }m\}\cup\{\infty\}$.
To see that Assumption \ref{NTZ} is satisfied, 
it suffices to show that
the unique odd character $\chi\in\CG{G}$
has conductor $m$.
Let $\pi:\Gal(\Q(\zeta_m)/\Q)\to G$ be the natural projection.
Since $\ker(\chi)=\{1_G\}$, we have
$ker(\chi\circ\pi)=\ker(\pi)=\Gal(\Q(\zeta_m)/K)$.
On the other hand,
if $\chi$ has conductor $f<m$,
then we must have $\Gal(\Q(\zeta_m)/\Q(\zeta_f))\subset\ker(\chi\circ\pi)$.
This implies $K\subset\Q(\zeta_f)$.
However, this contradicts the fact that $K$ has conductor $m$.
\item
The case where
$K=\Q(\zeta_{p^aq^b})$ and $S=\{p,q,\infty\}$
for distinct odd primes $p$ and $q$
satisfying
$$-1\in\braket{p}\subset(\Z/q\Z)^\times
\quad\text{and}\quad
-1\in\braket{q}\subset(\Z/p\Z)^\times,$$
and positive integers $a$ and $b$.
The proof is given in Lemma \ref{eg_LEM}.
\end{enumerate}
\end{rmk}
Next, we calculate the value $\Rks / R_{K^+,S}$.
Let $\mu_K$ be the group of roots of unity in $K$ and 
let 
$$\mathcal{Q}_{K,S}:=[\UG{K,S}:\mu_K\UG{K^+,S}]$$
be the $S$-unit index.
\begin{prop}\label{RKS}
Under Assumption \ref{NTZ}, we have
$$
\frac{\Rks}{R_{K^+,S}}=\frac{2^r}{\mathcal{Q}_{K,S}},
$$
where $r$ is defined in (\ref{r}).
\end{prop}
\begin{proof}
By Remark \ref{rmk_NTZ},
we have $r=\# S_K-1 = \#S_{K^+}-1$.
Let $\varepsilon=\{\varepsilon_1,\ldots,\varepsilon_r\}$ be a basis of
the free part of $\UG{K^+,S}$.
We choose $r$ elements of $S_{K^+}$ and denote them by $v_1,\ldots,v_r$.
Then the $S$-regulator of $K^+$ is defined to be
$$
R_{K^+,S}:=R_{K^+,S}(\varepsilon)
:=\mid \det(\log|\varepsilon_j|_{v_i})_{1\leq i,j\leq r}\mid.
$$
Let $\eta=\{\eta_1,\ldots,\eta_r\}$ be a basis of the free part of $\UG{K,S}$.
For $i=1,\ldots,r$, we choose a place $w_i$ of $K$ lying above $v_i$.
Then the $S$-regulator of $K$ is defined to be
$$
R_{K,S}:=R_{K,S}(\eta)
:=\mid \det(\log|\eta_j|_{w_i})_{1\leq i,j\leq r}\mid.
$$
We consider
$R_{K,S}(\varepsilon):=\mid \det(\log|\varepsilon_j|_{w_i})_{1\leq i,j\leq r}\mid$.
For $i=1,\ldots,r$, the places $v_i$ (finite or infinite) do not split in $K$.
Therefore, we have 
$\log|\varepsilon_j|_{w_i}=2\log|\varepsilon_j|_{v_i}$
for each $i$ and $j$.
Hence, we have
$$
R_{K,S}(\varepsilon)=2^r\cdot R_{K^+,S}.
$$
Moreover, by the same argument as in the proof of \cite[Lemma 4.15]{W},
we have 
$R_{K,S}(\varepsilon)=\mathcal{Q}_{K,S}\cdot R_{K,S}$.
This completes the proof.
\end{proof}

Therefore,
the formula (\ref{mSCNF}) can be written as follows.
\begin{thm}[``the minus $S$-class number formula'']\label{mSCNF2}
Under Assumption~\ref{NTZ}, we have
$$
\prod_{\substack{\chi\in\CG{G}\\\chi:\mathrm{odd}}}L_S(0,\chi)
=\frac{2^{r+1}}{w_K\cdot\mathcal{Q}_{K,S}}\hks^-.
$$
\end{thm}
\begin{proof}
By Remark~\ref{rmk_NTZ},
under Assumption~\ref{NTZ},
we have $L_S(0,\chi)=L_S^*(0,\chi)$
for an odd character $\chi\in\CG{G}$.
Hence, by (\ref{mSCNF}) and Proposition~\ref{RKS},
we obtain the theorem.
\end{proof}
\subsection{Stickelberger elements}\label{SE}
Let $m$ be an integer such that $m>2$ and assume that $4\mid m$ if $2\mid m$.
Let $\zeta_m$ be a primitive $m$-th root of unity.
In this subsection, we consider the cyclotomic field $\Q(\zeta_m)/\Q$ with Galois group $G_m$.
Let $S$ be the finite set of places of $\Q$ consisting of the primes dividing $m$ and $\infty$.
The Stickelberger element is defined to be
$$
\theta_m=\sum_{\substack{1\leq a\leq m\\(a,m)=1}}\left(\frac{1}{2}-\frac{a}{m}\right)\sigma_a^{-1} \in\Q[G_m]
$$
where $\sigma_a$ is the element in $G_m$ such that $\sigma_a(\zeta_m)=\zeta_m^a$.
For a non-trivial character $\chi\in \CG{G_m}$, we have
\begin{equation}\label{chi_theta1}
\chi^{-1}(\theta_m)
=-\frac{1}{m}\sum_{\substack{1\leq a\leq m\\(a,m)=1}}a\chi(\sigma_a).
\end{equation}
\begin{prop}\label{chi_theta2}
For a character $\chi\in \CG{G_m}$,
we have
$$
\chi^{-1}(\theta_m)=L_S(0,\chi).
$$
Hence, for all odd characters $\chi\in\CG{G_m}$,
we have
$\chi^{-1}(\theta_m)\ne0$
under Assumption~\ref{NTZ}.
\end{prop}
\begin{proof}
Theorem 4.1 and Proposition 4.2 of \cite{W} hold
without the assumption that $\chi$ is primitive.
The assertion follows from these results.
\end{proof}
\section{Bowen--Franks groups}\label{_BFG}
In this section, we first review some basic properties of
directed graphs (digraphs, for short), following \cite{LMV}.
The key result here is Theorem \ref{LMV2.3} regarding the cardinality of the torsion part of the Bowen--Franks group.
It is also important that
a decomposition similar to (\ref{decomp}) exists
on the graph side.
Finally, we prove two theorems
concerning the Bowen--Franks groups.
These are the main theorems of the present paper
(Theorems \ref{MTH} and \ref{__MTH2}).
\subsection{Digraphs}\label{__DG}
A digraph $X=(V_X, E_X)$ consists of a set of vertices $V_X$ and a set of directed edges $E_X$
with an incidence map
$$
\mathrm{inc}:E_X\to V_X\times V_X
\quad;\quad
e\mapsto(o(e),t(e)),
$$
where
$o(e)$ and $t(e)$
are called
the origin and the terminal vertex
of $e$ respectively.
For each vertex $v\in V_X$,
we set $E_{X,v}:=\{e\in E_X \mid o(e)=v\}$.
A digraph $X$ is said to be finite 
if both $V_X$ and $E_X$ are finite sets, and
it is said to be strongly connected 
if for any two vertices $u$ and $v$, there exists a directed path from $u$ to $v$.
In this paper,
we focus only on strongly connected finite digraphs.

Given a strongly connected finite digraph $X$, let $\Z V_X$ be the free abelian group on $V_X$.
The adjacency operator $\mathcal{A}_X:\Z V_X\rightarrow\Z V_X$ is defined by
$$
\mathcal{A}_X(v)=\sum_{e\in E_{X,v}}t(e)
$$
for $v\in V_X$.
The identity operator on $\Z V_X$ is denoted by $\mathcal{I}_X$.
\begin{dfn}
Let $X$ be a strongly connected finite digraph.
We define the Bowen--Franks operator on $\Z V_X$ by
$\mathcal{BF}_X:=\mathcal{I}_X-\mathcal{A}_X$.
Then the Bowen--Franks group of $X$ is defined by
$$
\mathrm{BF}(X):=\mathrm{coker}(\mathcal{BF}_X).
$$
\end{dfn}
\begin{dfn}
Let $X$ be a strongly connected finite digraph.
We set
$$
g_X(u):=\det(\mathcal{I}_X-\mathcal{A}_X u ) \ \in\Z[u],
$$
which is the inverse of the Bowen--Lanford zeta function (see \cite[\S 2.2]{LMV}).
\end{dfn} 
We set
$$
r_X:=\ord{u=1}(g_X(u)),\quad
s_X:=\rank_\Z(\BF{X}),\quad\text{and}\quad 
\delta_X:=r_X-s_X.
$$
Note that $r_X$ and $s_X$ are the algebraic and geometric multiplicities of $\mathcal{A}_X$ at the eigenvalue one respectively.
Hence, we have $\delta_X\geqq0$.
Let $\mathcal{A}_{X,\C}$ be the base change of $\mathcal{A}_X$ to $\C$.
If $\lambda_1,\ldots,\lambda_n$ are the eigenvalues of $\mathcal{A}_{X,\C}$, then we have
\begin{equation}\label{g}
g_X(u)=\prod_{i=1}^n(1-\lambda_i u)
\end{equation}
by the definition of $g_X(u)$.
From this equation, we have
\begin{equation}\label{g^*}
g_X^*(1):=\lim_{u\to1}\frac{g_X(u)}{(u-1)^{r_X}}=(-1)^{r_X}
\prod_{\lambda_i\ne1}(1-\lambda_i),
\end{equation}
where the product in the last term is taken over all eigenvalues that are not equal to one.
Lei, M$\mathrm{\ddot{u}}$ller, and Valli$\mathrm{\grave{e}}$res
proved the following theorem relating the special value $g_X^*(1)$ to the cardinality of the torsion part of the Bowen--Franks group, assuming $\delta_X=0$.

\begin{thm}[{\cite[Theorem 2.6]{LMV}}]\label{LMV2.3}
Let $X$ be a strongly connected finite digraph.
If $\delta_X=0$, then there exists a nonzero integer $m(X)$ such that
$$
g_X^*(1)=m(X)\cdot\#\BFtors{X}.
$$
\end{thm}
Next, we review the equivariant digraph theory.
Let $X$ be a strongly connected finite digraph, and let $G$ be a finite abelian group.
Recall that a voltage assignment is simply a function $\alpha:E_X\to G$.
We denote by $Y:=X(G,\alpha)$ the derived graph whose sets of vertices and edges are defined by 
$V_Y=V_X\times G$ and $E_Y=E_X\times G$ respectively 
(see \cite[\S 4.1]{LMV}).
We assume that $Y$ is also strongly connected.
Throughout this paper, we only consider derived graphs
that are strongly connected.
Then, by \cite[Proposition 4.2]{LMV},
the cover $Y/X$ is a Galois cover
(see \cite[Definition A.2]{LMV}).
Note that we can identify $\Z V_Y$ with $\Z[G]V_X$.
Since we can view the operator $\mathcal{I}_Y-\mathcal{A}_Yu$ as an operator
on a free $\Z[G][u]$-module of finite rank $\# V_X$,
we can define the following determinant.
\begin{dfn}
Let $Y/X$ be a Galois cover of strongly connected finite digraphs.
We set
$$\textstyle
\gamma_{Y/X}(u):=\det_{\Z[G][u]}(\mathcal{I}_Y-\mathcal{A}_Yu)
\in\Z[G][u],
$$
which is the inverse of the equivariant Artin--Mazur zeta function (see \cite[\S3.3]{LMV}).
\end{dfn}
A character $\chi\in\CG{G}$ induces a $\C$-ring homomorphism
$\C[G]\to \C$, which in turn induces a $\C[u]$-ring homomorphism
$\C[G][u]\to\C[u]$.
We denote this map by the same symbol $\chi$.
We set
$$
g_{Y/X}(u,\chi):=\chi(\gamma_{Y/X}(u)),
$$
and its leading coefficient at $u=1$ is denoted by $g_{Y/X}^*(1,\chi)$.
By \cite[Theorem 3.13]{LMV}, we have a decomposition
\begin{equation}\label{g_decomp}
g_Y(u)=\prod_{\chi\in\CG{G}}g_{Y/X}(u,\chi)
\end{equation}
and hence we have
$$
g_Y^*(1)
=\prod_{\chi\in\CG{G}}g_{Y/X}^*(1,\chi).
$$
\subsection{Bowen--Franks groups and minus class numbers}\label{__MTH}
Let $K/\Q$ be an imaginary abelian extension with
conductor $m$ and Galois group $G$.
Let $S$ be the finite set consisting of
the primes dividing $m$ and $\infty$.
The goal of this subsection is to prove the following theorem.
\begin{thm}\label{MTH}
Suppose that Assumption \ref{NTZ} holds.
Let $t$ be the number of finite places in $S_{K^+}$
which are inert in $K$.
Then there exists a digraph $Y$ such that
$$
\#\BFtors{Y}=m^{\frac{[K:\Q]}{2}+1}\hspace{1pt}2^{t+1}
\frac{[\Q(\zeta_m):K]}{w_K\cdot\mathcal{Q}_K}
h_K^-,
$$
where 
$h_K^-$ and $w_K$ denote 
the minus class number of $K$ and
the number of roots of unity in $K$
respectively,
and 
$\mathcal{Q}_K:=[\UG{K}:\mu_K\UG{K^+}]\in\{1,2\}$ is
the Hasse unit index.
\end{thm}

Let $X$ be a bouquet digraph
(meaning a digraph with a single vertex)
such that
$$
E_X=\{e_0, e_{i,j_i}\ |\ 
1\leq i< m,\ (i,m)=1,\ j_i=1,\ldots ,i\}.
$$
We define a voltage assignment $\alpha:E_X\to G_m:=\Gal(\Q(\zeta_m)/\Q)$ by
$$\alpha(e_0)=\sigma_1^{-1}\hspace{10pt}\text{and}\hspace{10pt}
\alpha(e_{i,j_i})=\sigma_i^{-1}.$$
Let $\pi:G_m\to G$ be the natural projection.
We denote the induced homomorphism $\Q[G_m]\to\Q[G]$
also by $\pi$.
Let
$$
Y:=X(G,\pi\circ\alpha)\hspace{10pt}\mathrm{and}\hspace{10pt}
Z:=X(G_m,\alpha)
$$
be the derived digraphs.
In this case,
$Y$ and $Z$ are strongly connected and finite
by \cite[Proposition 4.1]{LMV}.
Moreover, we have $\delta_Y=\delta_Z=0$ by \cite[Proposition 4.7]{LMV}.
\begin{prop}\label{gamma_L}
For a non-trivial character $\chi\in\CG{G}$,
we have
$$
g_{Y/X}(1,\chi^{-1})=m\cdot L_S(0,\chi).
$$
\end{prop}
\begin{proof}
We first consider the digraph $Z$.
Since $X$ is a bouquet digraph, $\Z V_Z$ is a free $\Z[G_m]$-module of rank one.
Therefore, by \cite[Corollary 4.6]{LMV}, we have
$$
\gamma_{Z/X}(1)
={\textstyle\det_{\Z[G_m]}(\mathcal{I}_Z-\mathcal{A}_Z)}
=1-\sum_{e\in E_X}\alpha(e)
=-\sum_{\substack{1\leq a\leq m\\(a,m)=1}}a\sigma_a^{-1}.
$$
Note that we can regard $\CG{G}$ as a subgroup of $\CG{G_m}$
via $\chi\mapsto\chi\circ\pi$.
By (\ref{chi_theta1}) and Proposition \ref{chi_theta2},
for a non-trivial character $\chi\in\CG{G}$,
we have
$$
\chi^{-1}\circ\pi(\gamma_{Z/X}(1))=
m\cdot\chi^{-1}\circ\pi(\theta_m)
=m\cdot L_S(0,\chi\circ\pi)
=m\cdot L_S(0,\chi),
$$
where $\theta_m$ is the Stickelberger element defined in \S\ref{SE}.
By the definition of $Y$,
we have
$$
\gamma_{Y/X}(1)
=1-\sum_{e\in E_X}\pi\circ\alpha(e)
=\pi(\gamma_{Z/X}(1)).
$$
By definition, we have $g_{Y/X}(1,\chi^{-1})=\chi^{-1}(\gamma_{Y/X}(1))$.
Hence, we obtain the proposition.
\end{proof}

Let $K^+$ and $\Q(\zeta_m)^+$ be the maximal real subfields
of $K$ and $\Q(\zeta_m)$ respectively.
We set $G^+:=\Gal(K^+/\Q)$ and $G_m^+:=\Gal(\Q(\zeta_m)^+/\Q)$.
Let $\pi^+:G\to G^+$ and $\pi_m^+:G_m\to G_m^+$ be
the natural projections.
We also denote the induced homomorphisms
$\Q[G]\to\Q[G^+]$ and $\Q[G_m]\to\Q[G_m^+]$
by $\pi^+$ and $\pi_m^+$ respectively.
Let
$$
Y^+:=X(G^+,\pi^+\circ\pi\circ\alpha)\hspace{10pt}\text{and}\hspace{10pt}
Z^+:=X(G_m^+,\pi_m^+\circ\alpha)
$$
be the derived digraphs,
which are strongly connected and finite.
Moreover, we have $\delta_{Y^+}=\delta_{Z^+}=0$.
\begin{prop}\label{mSCNF3}
Under Assumption \ref{NTZ},
we have
$$
\frac{m(Y)\cdot\#\BFtors{Y}}{m(Y^+)\cdot\#\BFtors{Y^+}}
=m^{\frac{\# G}{2}}\frac{2^{r+1}}{w_K\cdot\mathcal{Q}_{K,S}}\ \hks^-,
$$
where $r$ is defined in (\ref{r}).
\end{prop}
\begin{proof}
By Remark \ref{rmk_NTZ} and Proposition \ref{gamma_L},
for any odd character $\chi\in \CG{G}$,
we have
$g_{Y/X}(1,\chi^{-1}) \ne 0$
under Assumption \ref{NTZ}.
Therefore, we have
$$
g_{Y/X}^*(1,\chi^{-1}) = m\cdot L_S(0,\chi).
$$
By taking the product over all odd characters $\chi\in\CG{G}$,
we have
\begin{equation}\label{gl}
\prod_{\substack{\chi\in\CG{G}\\\chi:\mathrm{odd}}}g_{Y/X}^*(1,\chi)
=m^{\frac{\# G}{2}}
\prod_{\substack{\chi\in\CG{G}\\\chi:\mathrm{odd}}}L_S(0,\chi).
\end{equation}
The left hand side is equal to
$$
\frac{\prod_{\chi\in\CG{G}}g_{Y/X}^*(1,\chi)}
{\prod_{\chi\in\CG{G^+}}g_{Y^+/X}^*(1,\chi)}
=\frac{g_Y^*(1)}{g_{Y^+}^*(1)}
=\frac{m(Y)\cdot\#\BFtors{Y}}{m(Y^+)\cdot\#\BFtors{Y^+}}
$$
by (\ref{g_decomp}) and Theorem \ref{LMV2.3}.
Now, applying Theorem \ref{mSCNF2} to the right hand side of (\ref{gl})
completes the proof.
\end{proof}

We now calculate the three values
$\#\BFtors{Y^+}$, $m(Y^+)$, and $|m(Y)|$
following \cite[\S5.3]{LMV}.
The latter two values can be obtained by
the same argument;
that is, we have
\begin{equation}\label{m}
m(Y^+)=(-1)^{\frac{\# G}{2}}\cdot\frac{\# G}{2}
\quad\text{and}\quad
|m(Y)|=2^{\frac{\# G}{2}-1}\cdot\frac{\# G}{2}.
\end{equation}
For the value $\#\BFtors{Y^+}$,
an additional factor $[\Q(\zeta_m):K]$ appears in our setting.
By an argument similar to that in \cite[Proposition 5.3]{LMV},
we can show the following:
\begin{equation}\label{LMV5.3}
\gamma_{Z^+/X}(1)=-m\cdot N_{G_m^+}\in\Z[G_m^+].
\end{equation}
\begin{lem}\label{norm}
We have
$$
\gamma_{Y^+/X}(1)=-m[\Q(\zeta_m):K]\cdot N_{G^+}\in\Z[G^+].
$$
\end{lem}
\begin{proof}
We also denote by $\pi$ the induced projection $G_m^+\to G^+$. 
Since $\pi^+\circ\pi=\pi\circ\pi_m^+$,
it follows from (\ref{LMV5.3}) that we have
\begin{align*}
\gamma_{Y^+/X}(1) &= \pi^+(\gamma_{Y/X}(1))\\
  &= \pi^+\circ\pi(\gamma_{Z/X}(1))\\
  &= \pi\circ\pi_m^+(\gamma_{Z/X}(1))\\
  &= \pi(\gamma_{Z^+/X}(1))\\
  &= -m\cdot\pi(N_{G_m^+})\\
  &= -m[\Q(\zeta_m):K]\cdot N_{G^+},
\end{align*}
as desired.
\end{proof}
\begin{prop}\label{BFY^+}
We have an isomorphism
$$
\BF{Y^+}\cong\Z^{\frac{\# G}{2}-1}\oplus\Z/m[\Q(\zeta_m):K]\Z
$$
of $\Z$-modules.
In particular, we have
$\#\BFtors{Y^+}=m[\Q(\zeta_m):K]$.
\end{prop}
\begin{proof}
By the same argument as in the proof of \cite[Proposition 5.4]{LMV},
we have a split exact sequence
$$
0\to N_{G^+}\Z[G^+]/(\gamma_{Y^+/X}(1))
\to\BF{Y^+}
\to\Z[G^+]/N_{G^+}\Z[G^+]
\to0.
$$
Combining this with Lemma \ref{norm}
completes the proof.
\end{proof}

\begin{thm}\label{subMTH}
Under Assumption \ref{NTZ},
there exists a digraph $Y$ such that
$$
\#\BFtors{Y}=m^{\frac{\# G}{2}+1}\hspace{1pt}2^{\#(S_{K,\mathrm{fin}})+1}\ 
\frac{[\Q(\zeta_m):K]}{w_K\cdot\mathcal{Q}_{K,S}}\ \hks^-
$$
where 
$\hks^-, \mathcal{Q}_{K,S}$, and $w_K$ denote 
the minus $S$-class number of $K$,
the $S$-unit index, and 
the number of roots of unity in $K$
respectively.
\end{thm}
\begin{proof}
Note that the right hand side of Proposition \ref{mSCNF3} is positive.
By (\ref{m}) and Propositions \ref{mSCNF3} and \ref{BFY^+},
we have
$$
\#\BFtors{Y}
=m^{\frac{\# G}{2}+1}\hspace{1pt}2^{r-\frac{\# G}{2}+2}\frac{[\Q(\zeta_m):K]}
{w_K\cdot\mathcal{Q}_{K,S}}\ \hks^-.
$$
Since $r=\# S_K-1=\frac{\# G}{2} + \# S_{K,\mathrm{fin}}-1$,
this completes the proof of the theorem.
\end{proof}
\begin{proof}[Proof of Theorem \ref{MTH}] 
First, by \cite[Lemma 2.1]{Tate}, we have
$$
\Rks = R_K\left(\prod_{\pe\in S_{K,\mathrm{fin}}}\log N(\pe)\right)\frac{h_K}{\hks}
$$
where $R_K$ is the usual regulator.
Therefore, we have
$$
\frac{\Rks}{R_{K^+,S}}=\frac{R_K}{R_{K^+}}\cdot
\frac{\prod_{\pe\in S_{K,\mathrm{fin}}}\log N(\pe)}
{\prod_{\pe\in S_{K^+,\mathrm{fin}}}\log N(\pe)}
\cdot\frac{h_K^-}{\hks^-}
=2^t\cdot\frac{R_K}{R_{K^+}}
\cdot\frac{h_K^-}{\hks^-}
$$
where $t$ denotes the number of finite places in $S_{K^+}$
which are inert in $K$.
Let $S_{K,\infty}$ be the set of infinite places of $K$.
By \cite[Proposition 4.16]{W}, we have
$$
\frac{R_K}{R_{K^+}}=\frac{1}{\mathcal{Q}_K}2^{\# S_{K,\infty}-1}.
$$
Combining the above results with Proposition \ref{RKS}, we have
$$
\frac{\hks^-}{\mathcal{Q}_{K,S}}
=2^{-\# S_{K,\mathrm{fin}}+t}
\frac{h_K^-}{\mathcal{Q}_K}.
$$
By Theorem \ref{subMTH}, we obtain the theorem.
\end{proof}

Finally, we present special cases as corollaries.
\begin{cor}\label{corMTH}
Let $K$ be an imaginary quadratic field
with conductor $m$.
Then, there exists a digraph $Y$ such that
$$
\#\BFtors{Y}=m^2
\frac{\varphi(m)}{w_K}h_K,
$$
where $\varphi$ denotes Euler's totient function.
\end{cor}
\begin{proof}
In this case,
Assumption \ref{NTZ} is satisfied
by Remark \ref{rmk_NTZ_holds}. 
Since all finite places in $S_{K^+}=S_{\Q}$ are totally ramified in $K$,
we have $t=0$.
Moreover, we have 
$\mathcal{Q}_K=1$, $[\Q(\zeta_m):K]=\varphi(m)/2$, and $h_K^-=h_K$.
Hence, by Theorem \ref{MTH}, we obtain the corollary.
\end{proof}
\begin{cor}\label{corcycMTH}
For an odd rational prime $p$ and
a positive integer $n$,
let $K$ be a cyclotomic field
with conductor $p^n$.
Then, there exists a digraph $Y$ such that
$$
\#\BFtors{Y}=p^\frac{np^{n-1}(p-1)}{2}h_K^{-}.
$$
\end{cor}
\begin{proof}
In this case,
Assumption \ref{NTZ} is satisfied
by Remark \ref{rmk_NTZ_holds}. 
It is clear that $[\Q(\zeta_{p^n}):K]=1$ and $w_K=2p^n$.
By \cite[Corollary 4.13]{W},
we have $\mathcal{Q}_K=1$.
Moreover,
since $p$ is totally ramified in $K$,
we have $t=0$.
Hence, by Theorem \ref{MTH}, we obtain the corollary.
\end{proof}
\subsection{The character components of the Bowen--Franks groups}\label{__character components}
Let $K/\Q$ be an imaginary abelian extension with conductor $m$ and Galois group $G$.
We denote by $S$ the finite set consisting of the primes dividing $m$ and $\infty$,
and set $S_\mathrm{fin}:=S\setminus\{\infty\}$.
Let $Y/X$ be the Galois cover of digraphs with Galois group $G$
constructed in $\S$\ref{__MTH}.
Let $\psi$ be an odd character of $G$
with conductor $f_\psi$.
In this subsection,
for any $\ell\in S_\mathrm{fin}$,
we denote by $\psi(\ell)$ the value at $\ell$ of the primitive character associated with $\psi$, that is, $\psi(\ell)=\psi^{\mathrm{prim}}(\ell)$.
Fix an odd rational prime $p$.
Let $\Delta$ be the maximal subgroup of $G$
whose order is prime to $p$.
When $\Q(\zeta_p)\subset K$,
let $\omega:\Delta\to\Z_p^\times$
be the character determined by
$\sigma(\zeta_p)=\zeta_p^{\omega(\sigma)}$
for $\sigma\in\Delta$,
where the exponent is understood modulo $p$.
We call $\omega$ the Teichm\"uller character.
Throughout this subsection,
we impose the following two assumptions.
Note that Assumption \ref{NTZpsi} is
the analogue of Assumption \ref{NTZ}
for the fixed odd character $\psi$.
\begin{asmp}\label{NTZpsi}
For the fixed $\psi$ and any $\ell\in S_\mathrm{fin}$, we have $\psi(\ell)\ne1$.
\end{asmp}
\begin{asmp}\label{psineomega}
If $\Q(\zeta_p)\subset K$,
we have $\psi|_\Delta\ne\omega$.
\end{asmp}
In this subsection,
we study the relationship between
the $\psi$-components of the $p$-adic completion of the Bowen--Franks group and
the ideal class group of $K$.
We denote by
$$\BFp{Y}:=\BF{Y}\otimes_\Z\Z_p
\quad\text{and}\quad
A_K:=\Cl{K}\otimes_\Z\Z_p$$
the $p$-adic completion of the Bowen--Franks group of $Y$
and the ideal class group of $K$ respectively.
These are $\Z_p[G]$-modules.
Set $O_\psi:=\Z_p[\psi(G)]$.
Note that $O_\psi$ is a discrete valuation ring.
Let $\varpi$ be a uniformizer of $O_\psi$,
and let $v_\psi$ be the corresponding valuation
normalized by $v_\psi(\varpi)=1$.
We define the $\psi$-components of $\BFp{Y}$ and $A_K$ by
$$\BFopsi{Y}:=\BFp{Y}\otimes_{\Z_p[G]}O_\psi
\quad\text{and}\quad
A_K^\psi:=A_K\otimes_{\Z_p[G]}O_\psi$$
respectively, where $G$ acts on $O_\psi$ via $\psi$.
The goal of this subsection is to prove the following theorem,
which is the second main result of the present paper.
\begin{thm}\label{__MTH2}
Suppose that Assumptions \ref{NTZpsi} and \ref{psineomega} hold.
Let $T_\psi$ be the subset of $S_\mathrm{fin}$ defined by
$$
T_\psi:=\{\ \ell\in S_\mathrm{fin}\ \mid\ 
\text{the order of $\psi(\ell)$ is a power of $p$},\ 
\ell\nmid f_\psi,\ p\nmid \#I_\ell\ \},
$$
where $I_\ell$ denotes the inertia subgroup of $\ell$ in $G$.
We set
$$
c_\psi:=\sum_{\ell\in T_\psi}v_\psi(1-\psi^{-1}(\ell)).
$$
We have
$$\#\BFopsi{Y}
=\mid m\varpi^{c_\psi}\mid_p^{-\rank_{\Z_p}O_\psi}\# A_K^\psi,$$
where $|\cdot|_p$ denotes the usual $p$-adic absolute value.
In particular, if $p\nmid\# G$, then we have
$$\#\BFopsi{Y}
=p^\delta\# A_K^\psi,$$
where $\delta=0$ if $p\nmid m$, and $\delta=\rank_{\Z_p}O_\psi$ if $p\mid m$.
\end{thm}

Following the argument in the proof of \cite[Theorem 5.11]{LMV},
we first study $\BFopsi{Y}$.
Let $B_{n,\psi}$ denote the $n$-th generalized Bernoulli number
of the primitive character associated with $\psi$.
We also denote by $L(s,\psi)$ the $L$-function
of the primitive character associated with $\psi$.
Then it is known that
$$-B_{1,\psi}=L(0,\psi).$$
\begin{prop}\label{BFo}
For an odd character $\psi\in\CG{G}$ and
each rational prime $\ell$
such that $\ell\mid m$ and $\ell\nmid f_\psi$,
we set
$$b_\ell:=v_\psi(1-\psi^{-1}(\ell))
\quad\text{and}\quad
b_\psi:=\sum_{\ell\mid m,\ell\nmid f_\psi}b_\ell.
$$
Then, we have
$$\BFopsi{Y}
\ \cong\  O_\psi/\psi(\gamma_{Y/X}(1))O_\psi
\ =\ O_\psi/m B_{1,\psi^{-1}}\varpi^{b_\psi}O_\psi
.$$
\end{prop}
\begin{proof}
By the definition of $\BF{Y}$,
we have the exact sequence
$$\Z V_Y\xrightarrow{\mathcal{BF}_Y}\Z V_Y\longrightarrow\BF{Y}\longrightarrow0.$$
In the present setting,
since the base graph $X$ is a bouquet,
$\Z V_Y$ is isomorphic to $\Z[G]$.
By base change to $O_\psi$,
we obtain the following exact sequence:
$$O_\psi\xrightarrow{\mathcal{BF}^\psi}O_\psi\longrightarrow\BFopsi{Y}\longrightarrow0,$$
where the induced map $\mathcal{BF}^\psi$
is given by multiplication by $\psi(\gamma_{Y/X}(1))$.
By Proposition \ref{gamma_L}, we have
$$\psi(\gamma_{Y/X}(1))
=m\cdot L_S(0,\psi^{-1})
=-m B_{1,\psi^{-1}}\prod_{\ell\mid m,\ell\nmid f_\psi}(1-\psi^{-1}(\ell)).$$
This completes the proof.
\end{proof}

Next, we study $A_K^\psi$.
Thanks to the remarkable theorem \cite[Theorem 0.1]{KM11},
we can determine the order of $A_K^\psi$.
\begin{thm}[{\cite[Theorem 0.1]{KM11}}]\label{thm0.1}
Let $K/\Q$ be a finite imaginary abelian extension
with Galois group $G$.
We set $R_K=\Z[1/2][G]$.
For any $R_K$-module $M$,
we denote by $M^-$ its minus part.
Then, we have
$$\Fitt_{R_K^-}\big((\Cl{K}\otimes_\Z\Z[1/2])^-\big)=\varTheta^-_K,$$
where the left hand side is
the (initial) Fitting ideal of
the $R_K^-$-module $(\Cl{K}\otimes_\Z\Z[1/2])^-$ and 
$\varTheta_K$ is the $R_K$-ideal defined in \cite[\S 1.2]{KM11},
called the Stickelberger ideal for $K$
(see \eqref{varTheta} and Remark \ref{rmk_varTheta}).
\end{thm}

As a consequence of Theorem \ref{thm0.1},
we obtain the following.
\begin{cor}\label{cor_Fitt}
In the setting of Theorem \ref{thm0.1},
let $p$ be an odd rational prime and
let $\psi$ be an odd character of $G$.
Then we have
$$
\Fitt_{O_\psi}(A_K^\psi)=\psi(\varTheta_K\otimes\Z_p),
$$
where we also denote by $\psi$ the induced homomorphism
$\Z_p[G]\to O_\psi$, and set
$\varTheta_K\otimes\Z_p:=\varTheta_K\otimes_{\Z[1/2]}\Z_p\cong\varTheta_K\otimes_{R_K}\Z_p[G]$.
\end{cor}
\begin{proof}
Since $p$ is odd and $\psi$ is an odd character,
$\psi:\Z_p[G]\to O_\psi$
factors through the minus part.
Thus, we have ring homomorphisms
$$
R_K^-\to\Z_p[G]^-\to O_\psi.
$$
By the base change property of Fitting ideals,
we have
$$
\Fitt_{O_\psi}((\Cl{K}\otimes_\Z\Z[1/2])^-
\otimes_{R_K^-}O_\psi)
=\Fitt_{R_K^-}\big((\Cl{K}\otimes_\Z\Z[1/2])^-\big)O_\psi.
$$
Moreover, there is a natural isomorphism
$$
(\Cl{K}\otimes_\Z\Z[1/2])^-
\otimes_{R_K^-}O_\psi
\cong
(\Cl{K}\otimes_\Z\Z_p)
\otimes_{\Z_p[G]}O_\psi
=A_K^\psi.
$$
Therefore, Theorem \ref{thm0.1} gives
$$
\Fitt_{O_\psi}(A_K^\psi)
=\Fitt_{R_K^-}\big((\Cl{K}\otimes_\Z\Z[1/2])^-\big)O_\psi
=\varTheta_K^- O_\psi
=\psi(\varTheta_K\otimes \Z_p).
$$
This completes the proof.
\end{proof}

Since $O_\psi$ is a discrete valuation ring,
by Corollary \ref{cor_Fitt},
we have
\begin{equation}\label{A_K^psi}
\# A_K^\psi = \#\left(\frac{O_\psi}{\psi(\varTheta_K\otimes\Z_p)}\right).
\end{equation}

We now investigate $\psi(\varTheta_K\otimes\Z_p)$.
To this end, we recall the definition of $\varTheta_K$ in \cite{KM11}.
For a finite abelian extension $K'/\Q$ containing $K$,
we define the $\Z_p$-algebra homomorphism
$$c_{K'/K}:\Z_p[\Gal(K'/\Q)]\to\Z_p[G]$$
induced by the restriction $\sigma\mapsto\sigma|_K$
for any $\sigma\in\Gal(K'/\Q)$.
The $\Z_p$-module homomorphism 
$$\nu_{K'/K}:\Z_p[G]\to\Z_p[\Gal(K'/\Q)]$$
is defined by
$\sigma\mapsto\sum_{c_{K'/K}(\tau)=\sigma}\tau$
for any $\sigma\in G$.

Suppose that the conductor $m$ of $K$
has the following prime factorization:
$$m=\ell_1^{e_1}\cdots\ell_n^{e_n}\quad(e_1,\ldots,e_n>0).$$
For a rational prime $\ell\mid m$,
the inertia subgroup of $\ell$ in $G$ is denoted by $I_\ell$.
We make the following auxiliary assumption on $K$:
\begin{equation}\label{auxiliary assumption}
G=\Gal(K/\Q)=I_{\ell_1}\times\cdots\times I_{\ell_n}.
\end{equation}
This assumption is needed only for the computation
and can ultimately be removed (see Remark \ref{rmk_varTheta2} and Proposition \ref{psi_varTheta}).
We set
$$\mathcal{D}:=\{d=\ell_1^{e_1'}\cdots\ell_n^{e_n'}\mid e_i'= 0\text{ or }e_i,\ i=1,\ldots,n\}.$$
For any $d\in\mathcal{D}$,
we define $K_d$ to be the maximal subfield of $K$
which is unramified outside the primes dividing $d$ over $\Q$,
so that
$$\Gal(K_d/\Q)=\prod_{\ell\mid d}I_\ell.$$
For $d\in\mathcal{D}$, 
the Stickelberger element of $K_d$ is defined by
$$
  \theta_{K_d}=
    \begin{cases}
        \pi_d(\theta_d) & \text{if }d>1,\\
        -1/2 & \text{if }d=1,
    \end{cases}
$$
where, for $d>1$,
$\theta_d$ denotes the Stickelberger element
of $\Q(\zeta_d)$
defined in $\S$\ref{SE}
and $\pi_d$ denotes the natural projection
from $\Gal(\Q(\zeta_d)/\Q)$ to $\Gal(K_d/\Q)$.
Using the partial zeta function $\zeta(s,\sigma)$,
$\theta_{K_d}$ can also be written as
$$\theta_{K_d}=\sum_{\sigma\in\Gal(K_d/\Q)}\zeta(0,\sigma)\sigma^{-1}
\in\Q[\Gal(K_d/\Q)].$$

Let $\varTheta_K'$ be the $R_K$-submodule of $\Q[G]$
generated by all $\nu_{K/K_d}(\theta_{K_d})$ for $d\in\mathcal{D}$,
that is
\begin{equation}\label{varTheta'}
\varTheta_K'=\langle\nu_{K/K_d}(\theta_{K_d})\mid d\in\mathcal{D}\rangle
\subset \Q[G].
\end{equation}
Then the Stickelberger ideal for $K$ is defined by
\begin{equation}\label{varTheta}
\varTheta_K:=\varTheta_K' \cap R_K.
\end{equation}
\begin{rmk}\label{rmk_varTheta}
For $K$ not necessarily satisfying \eqref{auxiliary assumption},
the Stickelberger ideal $\varTheta_K$ is also defined in \cite{KM11}.
For such $K$, by \cite[Lemma 2.3]{K03}, 
there exists a unique abelian extension $K'/\Q$ containing $K$
such that $K'/K$ is unramified at all finite primes
and $K'/\Q$ satisfies \eqref{auxiliary assumption}.
Then, we define the Stickelberger ideal $\varTheta_K$ to be
$$\varTheta_K:=c_{K'/K}(\varTheta_{K'})\subset R_K.$$
\end{rmk}
\begin{rmk}\label{rmk_varTheta2}
Let $K$ and $K'$ be as in Remark \ref{rmk_varTheta}.
Set $G':=\Gal(K'/\Q)$. 
For an odd character $\psi\in\CG{G}$, we set
$\psi':=\psi\circ c_{K'/K}:\Z_p[G']\to O_\psi$.
Then, we have
$$\psi(\varTheta_K\otimes\Z_p)=\psi'(\varTheta_{K'}\otimes\Z_p).$$
We note that $K'$ also has conductor $m$.
For each $\ell\mid m$,
we denote by $I'_\ell$
the inertia subgroup of $\ell$ in $G'$.
Since $K'/K$ is unramified,
we have $I_\ell\cong I'_\ell$.
Moreover, since $c_{K'/K}$ is surjective,
we have $O_\psi=O_{\psi'}$.
By the definition of $\psi'$,
the primitive characters associated with $\psi$ and $\psi'$ coincide,
and hence
$B_{1,\psi^{-1}}=B_{1,\psi'^{-1}}$.

Let $\Delta'$ be the maximal subgroup of $G'$
whose order is prime to $p$.
When $\Q(\zeta_p)\subset K'$,
we denote by $\omega'$ the Teichm\"uller character on $\Delta'$.
We also note that Assumption \ref{psineomega} holds for $\psi$
if and only if it holds for $\psi'$.
Indeed, the restriction map induces a surjection
$\Delta'\to\Delta$, and
$$\psi'|_{\Delta'}
=(\psi|_\Delta)\circ(c_{K'/K}|_{\Delta'}).$$
If $\Q(\zeta_p)\subset K$, then
$\omega'=\omega\circ(c_{K'/K}|_{\Delta'})$
so the assertion follows
from the surjectivity of $\Delta'\to\Delta$.
If $\Q(\zeta_p)\not\subset K'$, then $\Q(\zeta_p)\not\subset K$,
and both assumptions hold automatically.
It remains to consider the case where
$\Q(\zeta_p)\not\subset K$ but $\Q(\zeta_p)\subset K'$.
Suppose that $\psi'|_{\Delta'}=\omega'$.
Since $\psi'|_{Gal(K'/K)}=1$,
we have
$$\omega'|_{Gal(K'/K)\cap\Delta'}=\psi'|_{Gal(K'/K)\cap\Delta'}=1.$$
Moreover, the $p$-Sylow subgroup of $G'$
acts trivially on $\zeta_p$.
It follows that every element of $\Gal(K'/K)$
fixes $\zeta_p$.
Thus $\zeta_p\in K$,
contrary to our assumption.
\end{rmk}

Write $G=\Delta\times\Gamma$, where $\Delta$ has order prime to $p$
and $\Gamma$ is the $p$-Sylow subgroup of $G$.
Under this decomposition, every $\psi\in\CG{G}$
can be written uniquely as
$\psi=\chi\varphi$,
where $\chi\in\CG{\Delta}$ and $\varphi\in\CG{\Gamma}$.
\begin{lem}\label{varTheta_VS_varTheta'}
Suppose that Assumption \ref{psineomega} and \eqref{auxiliary assumption} hold.
Then, we have
$$\psi(\varTheta_K\otimes\Z_p)=\psi(\varTheta_K'\otimes\Z_p).$$
\end{lem}
\begin{proof}
Set $\chi:=\psi|_\Delta$ and
$O_\chi:=\Z_p[\chi(\Delta)]$.
By \cite[Lemma 1.3]{KM11},
we have
$$
(\varTheta_K\otimes\Z_p)\otimes_{Z_p[G]} O_\chi[\Gamma]
=
\left((\varTheta'_K\otimes\Z_p)\otimes_{Z_p[G]} O_\chi[\Gamma]\right)
\cap O_\chi[\Gamma].
$$
If $\Q(\zeta_p)\not\subset K$, then
$\theta_{K_d}\in\Z_p[\Gal(K_d/\Q)]$
for every $d\in\mathcal{D}$.
Since $\varTheta_K'$
is generated by the elements
$\nu_{K/K_d}(\theta_{K_d})$,
it follows that
$\varTheta_K'\otimes\Z_p\subset\Z_p[G]$.
Hence, we have
\begin{equation}\label{Lem1.3}
(\varTheta_K\otimes\Z_p)\otimes_{Z_p[G]} O_\chi[\Gamma]
=(\varTheta'_K\otimes\Z_p)\otimes_{Z_p[G]} O_\chi[\Gamma].
\end{equation}
If $\Q(\zeta_p)\subset K$, then
Assumption \ref{psineomega} gives $\chi\ne\omega$.
Thus, \eqref{Lem1.3} also holds
by the discussion immediately following \cite[Lemma 1.3]{KM11}.

In both cases, taking the images of both sides under
$\varphi:O_\chi[\Gamma]\to O_\psi$,
we obtain the result.
\end{proof}

By \eqref{varTheta'} and Lemma \ref{varTheta_VS_varTheta'},
we are essentially reduced to determining $\psi(\nu_{K/K_d}(\theta_{K_d}))$.
\begin{prop}\label{psi_varTheta}
Suppose that Assumption \ref{psineomega} holds.
For each rational prime $\ell$ such that $\ell\mid m$ and $\ell\nmid f_\psi$,
we set
$$a_\ell:=\min\{v_\psi(\# I_\ell),b_\ell\}
\quad\text{and}\quad
a_\psi:=\sum_{\ell\mid m, \ell\nmid f_\psi}a_\ell,$$
where $b_\ell=v_\psi(1-\psi^{-1}(\ell))$ is as defined in Proposition \ref{BFo}.
Then, we have
$$\psi(\varTheta_K\otimes\Z_p)=B_{1,\psi^{-1}}\varpi^{a_\psi}O_\psi.$$
\end{prop}
\begin{proof}
By Remark \ref{rmk_varTheta2} and Lemma \ref{varTheta_VS_varTheta'},
it suffices to show the result
under the assumption \eqref{auxiliary assumption}.
We can compute $\psi(\nu_{K/K_d}(\theta_{K_d}))$ as follows:
\begin{align*}
\psi(\nu_{K/K_d}(\theta_{K_d}))
  &= \sum_{\sigma\in\Gal(K_d/\Q)}\zeta(0,\sigma)\psi(\nu_{K/K_d}(\sigma^{-1}))\\
  &= \sum_{\sigma\in\Gal(K_d/\Q)}\zeta(0,\sigma)\psi(\tilde{\sigma})^{-1}
  \sum_{\tau\in\Gal(K/K_d)}\psi(\tau)^{-1},
\end{align*}
where $\tilde{\sigma}\in G$ denotes a lift of $\sigma\in\Gal(K_d/\Q)$.
Hence, if $\psi|_{\Gal(K/K_d)}\ne1$, then $\psi(\nu_{K/K_d}(\theta_{K_d}))=0$.
Suppose that $\psi|_{\Gal(K/K_d)}=1$, that is, $f_\psi\mid d$.
Let $\psi_{K_d}$ be the unique character of $\Gal(K_d/\Q)$
satisfying $\psi=\psi_{K_d}\circ c_{K/K_d}$.
Since $\psi(\tilde{\sigma})=\psi_{K_d}(\sigma)$,
we have
$$\psi(\nu_{K/K_d}(\theta_{K_d}))=[K:K_d]\psi_{K_d}(\theta_{K_d}).$$
By \eqref{auxiliary assumption}, we have
$$[K:K_d]=\prod_{\ell\mid m, \ell\nmid d}\# I_\ell.$$
By Proposition \ref{chi_theta2}, we have
$$\psi_{K_d}(\theta_{K_d})
=L_{S_d}(0,\psi_{K_d}^{-1})
=L_{S_d}(0,\psi^{-1})
=-B_{1,\psi^{-1}}\prod_{\ell\mid d, \ell\nmid f_\psi}(1-\psi^{-1}(\ell)),$$
where $S_d:=\{\ell\mid d\}\cup\{\infty\}$.
Therefore, in this case, we have
\begin{equation}\label{psi_nu_theta}
\psi(\nu_{K/K_d}(\theta_{K_d}))
=-B_{1,\psi^{-1}}
\prod_{\ell\mid m, \ell\nmid d}\# I_\ell
\prod_{\ell\mid d, \ell\nmid f_\psi}(1-\psi^{-1}(\ell)).
\end{equation}
Hence, by \eqref{varTheta'} and \eqref{psi_nu_theta}, we have
\begin{align*}
\psi(\varTheta_K'\otimes\Z_p)
&=\left\langle\left. B_{1,\psi^{-1}}
\prod_{\ell\mid m, \ell\nmid d}\# I_\ell
\prod_{\ell\mid d, \ell\nmid f_\psi}(1-\psi^{-1}(\ell))
\ \right|\ d\in\mathcal{D},\ f_\psi\mid d\right\rangle\\
&=B_{1,\psi^{-1}}\prod_{\ell\mid m, \ell\nmid f_\psi}\left(\# I_\ell,\ 1-\psi^{-1}(\ell)\right).
\end{align*}
Since
$\left(\# I_\ell,\ 1-\psi^{-1}(\ell)\right)
=\varpi^{a_\ell}O_\psi$,
the result follows.
\end{proof}

\begin{lem}\label{lem_MTH2}
We assume Assumption \ref{NTZpsi}.
We have
$$c_\psi=b_\psi-a_\psi,$$
where $c_\psi$, $b_\psi$, and $a_\psi$
are as defined in Theorem \ref{__MTH2}, Propositions \ref{BFo} and \ref{psi_varTheta}
respectively.
\end{lem}
\begin{proof}
Recall that,
for each rational prime $\ell$
such that $\ell\mid m$ and $\ell\nmid f_\psi$,
we defined
$$
b_\ell:=v_\psi(1-\psi^{-1}(\ell))
\quad\text{and}\quad
a_\ell:=\min\{v_\psi(\# I_\ell),b_\ell\}.$$
We have
$$b_\ell-a_\ell
=\max\{b_\ell - v_\psi(\# I_\ell),\ 0\}.$$
We now investigate $b_\ell - a_\ell$.
Note that
$1-\psi^{-1}(\ell)\ne0$
by Assumption \ref{NTZpsi}.

First, if the order of $\psi^{-1}(\ell)$ is not a power of $p$,
then since $1-\psi^{-1}(\ell)$ is a unit in $O_\psi$,
we have $b_\ell=0$.
Therefore, $b_\ell - v_\psi(\# I_\ell)\leq0$ so $b_\ell-a_\ell=0$.

Next, we assume that
the order of $\psi^{-1}(\ell)$ is $p^r$ for some $r>0$.
In this case, we have
\begin{equation}\label{v}
b_\ell=\frac{v_\psi(p)}{p^{r-1}(p-1)}<v_\psi(p).
\end{equation}
If $p\mid\# I_\ell$,
then $v_\psi(\# I_\ell)\geq v_\psi(p)$.
Combining this with \eqref{v},
we have $b_\ell<v_\psi(\# I_\ell)$.
Therefore, $b_\ell-a_\ell=0$.
On the other hand,
if $p\nmid \# I_\ell$, then $v_\psi(\# I_\ell)=0$.
Therefore, $a_\ell=0$, so $b_\ell-a_\ell=b_\ell$.
The result follows from the definition of $c_\psi$.
\end{proof}

\begin{proof}[Proof of Theorem \ref{__MTH2}]
By Proposition \ref{BFo}, we have
\begin{equation}\label{MTH2_BFo}
\#\BFopsi{Y}=|m B_{1,\psi^{-1}}\varpi^{b_\psi}|_p^{-\rank_{\Z_p}O_\psi}.
\end{equation}
By \eqref{A_K^psi} and Proposition \ref{psi_varTheta},
we have
\begin{equation}\label{MTH2_A_K^psi}
\# A_K^\psi=|B_{1,\psi^{-1}}\varpi^{a_\psi}|_p^{-\rank_{\Z_p}O_\psi}.
\end{equation}
Therefore, by \eqref{MTH2_BFo} and \eqref{MTH2_A_K^psi},
we have
$$
\#\BFopsi{Y}=|m \varpi^{b_\psi-a_\psi}|_p^{-\rank_{\Z_p}O_\psi}\# A_K^\psi.
$$
Hence, the first statement follows immediately from Lemma \ref{lem_MTH2}.

Next, we assume that $p\nmid \# G$.
For $\ell\in S_\mathrm{fin}$ with $\ell\nmid f_\psi$,
Assumption \ref{NTZpsi} gives $\psi(\ell)\ne1$.
Thus, in this case,
the order of $\psi(\ell)$ is not a power of $p$,
and hence $T_\psi=\emptyset$.
Therefore, Lemma~\ref{lem_MTH2} gives
$$\#\BFopsi{Y}=|m|_p^{-\rank_{\Z_p}O_\psi}\# A_K^\psi.$$
If $p\nmid m$, then $|m|_p^{-1}=1$.
On the other hand, suppose $p\mid m$.
Since $p\nmid\# I_p$,
$p$ is tamely ramified in $K$.
Hence $|m|_p^{-1}=p$.
This completes the proof.
\end{proof}

\section{Examples}\label{_EG}
In this section,
we provide some examples for Theorem \ref{MTH}
and Corollaries \ref{corMTH} and \ref{corcycMTH}.
\subsection{Imaginary quadratic fields}\label{__EG1}
In this subsection,
we consider the case where
$K$ is an imaginary quadratic field.
In this case,
Assumption \ref{NTZ} is satisfied
by Remark \ref{rmk_NTZ_holds}.
First, we consider $K=\Q(\sqrt{-35})$.
Then, we have $m=35$.
We set $S:=\{5, 7, \infty\}$.
Let $G$ and $G_m$ be as in \S\ref{__MTH}. 
We have $w_K=2$.
It is known that $h_K=2$.
Therefore, the right hand side of Corollary \ref{corMTH} is
equal to $2^3\cdot3\cdot5^2\cdot7^2$.

We consider the graph side.
Let $X$ be the bouquet digraph such that
$$
V_X=\{v\}\hspace{10pt}\mathrm{and}\hspace{10pt}
E_X=\{e_0, e_{i,j_i}\ |\ 
1\leq i< 35,\ (i,35)=1,\ j_i=1,\ldots ,i\}.
$$
We define the voltage assignment $\alpha:E_X\to G_m$
by $\alpha(e_0)=\sigma_1^{-1}$ and $\alpha(e_{i,j_i})=\sigma_i^{-1}$.
In this case, the images of $\sigma_i\in G_m$ under
the natural projection $\pi:G_m\to G$
can be classified as follows:
\begin{equation}\label{classification}
\pi(\sigma_i)=\left\{ \,
					\begin{aligned}
						& 1_G\ &&(i=1,3,4,9,11,12,13,16,17,27,29,33)\\
						& \tau\ &&(i=2,6,8,18,19,22,23,24,26,31,32,34)
					\end{aligned}
				\right.
\end{equation}
where $\tau$ denotes the complex conjugation in $G$.
The derived digraph $Y:=X(G,\pi\circ\alpha)$ is given by
the following graph.
\begin{center}
\begin{tikzpicture}
\fill[white] (0,1) circle(0);
\fill[white] (0,-1) circle(0);
\def\t{48};
\coordinate (A) at (0,0);
\coordinate (B) at (2,0);
\coordinate (C) at (-1,0);
\coordinate (D) at (3,0);
\fill[black] (A) circle (0.08);
\fill[black] (B) circle (0.08);
\draw [thin] (A) to[out=\t, in=180-\t] (B);
\draw [thin] (A) to[out=-\t, in=\t-180] (B);
\draw [thin] (A) to[out=\t+90, in=90] (C);
\draw [thin] (A) to[out=-\t-90, in=-90] (C);
\draw [thin] (B) to[out=90-\t, in=90] (D);
\draw [thin] (B) to[out=\t-90, in=-90] (D);
\def\a{0.433};
\draw [-{Stealth[length=1.7mm]}, thin] (1,\a) to (1.1,\a);
\draw [-{Stealth[length=1.7mm]}, thin] (1,-\a) to (0.9,-\a);
\def\b{-0.995};
\def\c{0.03};
\def\d{-0.003};
\def\e{-0.02};
\draw [-{Stealth[length=1.7mm]}, thin] (\b,0+\c) to (\b,-0.1+\c);
\draw [-{Stealth[length=1.7mm]}, thin] (\d+3,0+\e) to (\d+3,0.1+\e);
\draw (-1,0) node[left]{\scriptsize{$176$}};
\draw (3,0) node[right]{\scriptsize{$176$}};
\draw (1,0.4) node[above]{\scriptsize{$245$}};
\draw (1,-0.4) node[below]{\scriptsize{$245$}};
\end{tikzpicture}
\end{center}
Here each edge is labeled with
a small number
indicating its multiplicity.
Hence, the matrix representations $A_Y$ and $B_Y$ of
the adjacency operator $\mathcal{A}_Y$
and the Bowen--Franks operator $\mathcal{BF}_Y$
with respect to the basis $\{(v,1_G), (v,\tau)\}$
are given by
$$
A_Y=
\begin{pmatrix}
176 & 245\\
245 & 176
\end{pmatrix}
\hspace{10pt}\mathrm{and}\hspace{10pt}
B_Y=
\begin{pmatrix}
-175 & -245\\
-245 & -175
\end{pmatrix}
$$
respectively.
Since the matrix $B_Y$ has rank two,
we have
$$
\#\BFtors{Y}=\#\coker(\mathcal{BF}_Y)=\mid\det(B_Y)\mid
=2^3\cdot3\cdot5^2\cdot7^2.
$$
This is the desired result.

Next, we consider a more general case
to see that the above calculation
is essentially parallel to
that for
the classical class number formula of
imaginary quadratic fields.
Let $K$ be an imaginary quadratic field
with conductor $m$.
We set $S:=\{p\mid p\text{ divides }m\}\cup\{\infty\}$.
Let $G$ and $G_m$ be as in $\S3.2$.
We define the Dirichlet character $\psi$ by the composition
$$
\psi:(\Z/m\Z)^\times\xrightarrow{\sim}
G_m\xrightarrow{\pi}
G\xrightarrow{\sim}
\{\pm1\}.
$$
For $j\in\{\pm1\}$, we set
$$
I_j:=\sum_{\substack{1\leq i\leq m,\ (i,m)=1,\\ \psi(i)=j}}i\ .
$$
Then the classical class number formula of
imaginary quadratic fields is given by
$$
h_K=-\frac{w_K}{2m}\sum_{i=1}^{m}i\psi(i)
=-\frac{w_K}{2m}(I_1-I_{-1}).
$$
Since $I_1+I_{-1}=m\phi(m)/2$,
we have
\begin{equation}\label{formula_iqf}
m^2\frac{\phi(m)}{w_K}h_K
=-I_1^2+I_{-1}^2.
\end{equation}
On the other hand,
since the classification \eqref{classification} is determined by
whether $\psi(i)=1$ or $\psi(i)=-1$,
the matrix $B_Y$ is given by
$$
-B_Y=
\begin{pmatrix}
I_1 & I_{-1}\\
I_{-1} & I_1
\end{pmatrix}.
$$
Therefore, we have
$$
\#\BFtors{Y}=\mid\det(B_Y)\mid
=\mid I_1^2-I_{-1}^2\mid.
$$
Combining this with \eqref{formula_iqf},
we obtain Corollary \ref{corMTH}.
\subsection{Cyclotomic fields}\label{__EG2}
In this subsection,
we consider the case where
$K$ is a cyclotomic field.
First, let $K$ be the cyclotomic field $\Q(\zeta_{12})=\Q(\sqrt{3},\zeta_4)$
with conductor $m=12$.
We set $S:=\{2,3,\infty\}$ and
set $G=G_m=\Gal(K/\Q)$.
Since $2\equiv-1\mod 3$ and $3\equiv-1\mod 4$,
in this case, Assumption \ref{NTZ} is satisfied.
We have $w_K=12$.
It is known that $h_K^-=1$.
By \cite[Corollary 4.13]{W},
we have $\mathcal{Q}_K=2$.

To study the ramification of primes in $S$,
we consider the cyclotomic polynomial
$\Phi_{12}=X^4-X^2+1$.
Since 
$$
\Phi_{12}\equiv(X^2+X+1)^2\mod2
\hspace{10pt}\text{and}\hspace{10pt}
\Phi_{12}\equiv(X^2+1)^2\mod3,
$$
both $2$ and $3$ are ramified with index two
in the extension $K/\Q$.
On the other hand,
since the discriminant of $K^+=\Q(\sqrt{3})$ is $12$,
both $2$ and $3$ are totally ramified in $K^+$.
Therefore, we have $t=2$.
Hence, the right hand side of Theorem \ref{MTH} is
equal to $2^6\cdot3^2$.

We consider the graph side.
Let $X$ be the bouquet digraph such that
$$
V_X=\{v\}\hspace{10pt}\mathrm{and}\hspace{10pt}
E_X=\{e_0, e_{i,j_i}\ |\ 
i=1, 5, 7, 11,\ j_i=1,\ldots ,i\}.
$$
We define the voltage assignment $\alpha$
as in \S\ref{__MTH}.
The derived digraph $Y:=X(G,\alpha)$ is given by the following graph.
\begin{center}
\begin{tikzpicture}
\fill[white] (0,1) circle(0);
\fill[white] (0,-1) circle(0);
\def\t{48};
\def\r{0.04};
\coordinate (A) at (0,0);
\coordinate (A1) at (0-\r,0+\r);
\coordinate (A2) at (0+\r,0-\r);
\coordinate (B) at (2,0);
\coordinate (B1) at (2-\r,0-\r);
\coordinate (B2) at (2+\r,0+\r);
\coordinate (C) at (2,2);
\coordinate (C1) at (2+\r,2-\r);
\coordinate (C2) at (2-\r,2+\r);
\coordinate (D) at (0,2);
\coordinate (D1) at (0+\r,2+\r);
\coordinate (D2) at (0-\r,2-\r);
\fill[black] (A) circle (0.08);
\fill[black] (B) circle (0.08);
\fill[black] (C) circle (0.08);
\fill[black] (D) circle (0.08);
\draw [thin] (A1) to (B2);
\draw [thin] (A2) to (B1);
\draw [thin] (B1) to (C2);
\draw [thin] (B2) to (C1);
\draw [thin] (C1) to (D2);
\draw [thin] (C2) to (D1);
\draw [thin] (D1) to (A2);
\draw [thin] (D2) to (A1);
\draw [thin] (A1) to (C2);
\draw [thin] (A2) to (C1);
\draw [thin] (B1) to (D2);
\draw [thin] (B2) to (D1);
\def\t{0.7};
\def\u{0.65};
\draw [-{Stealth[length=1.6mm]}, very thin] (A1) to ($(A1)!\t!(B2)$);
\draw [-{Stealth[length=1.6mm]}, very thin] (B1) to ($(B1)!\t!(A2)$);
\draw [-{Stealth[length=1.6mm]}, very thin] (B1) to ($(B1)!\t!(C2)$);
\draw [-{Stealth[length=1.6mm]}, very thin] (C1) to ($(C1)!\t!(B2)$);
\draw [-{Stealth[length=1.6mm]}, very thin] (C1) to ($(C1)!\t!(D2)$);
\draw [-{Stealth[length=1.6mm]}, very thin] (D1) to ($(D1)!\t!(C2)$);
\draw [-{Stealth[length=1.6mm]}, very thin] (D1) to ($(D1)!\t!(A2)$);
\draw [-{Stealth[length=1.6mm]}, very thin] (A1) to ($(A1)!\t!(D2)$);
\draw [-{Stealth[length=1.6mm]}, very thin] (A1) to ($(A1)!\t!(C2)$);
\draw [-{Stealth[length=1.6mm]}, very thin] (C1) to ($(C1)!\t!(A2)$);
\draw [-{Stealth[length=1.6mm]}, very thin] (B1) to ($(B1)!\t!(D2)$);
\draw [-{Stealth[length=1.6mm]}, very thin] (D1) to ($(D1)!\t!(B2)$);
\draw ($(A1)!\u!(B2)$)node[above]{\scriptsize$11$};
\draw ($(B1)!\u!(A2)$)node[below]{\scriptsize$11$};
\draw ($(B1)!\u!(C2)$)node[left]{\scriptsize$7$};
\draw ($(C1)!\u!(B2)$)node[right]{\scriptsize$7$};
\draw ($(C1)!\u!(D2)$)node[below]{\scriptsize$11$};
\draw ($(D1)!\u!(C2)$)node[above]{\scriptsize$11$};
\draw ($(D1)!\u!(A2)$)node[right]{\scriptsize$7$};
\draw ($(A1)!\u!(D2)$)node[left]{\scriptsize$7$};
\draw (0.85,0.5)node{\scriptsize$5$};
\draw (1.5,0.8)node{\scriptsize$5$};
\draw (1.25,1.55)node{\scriptsize$5$};
\draw (0.5,1.2)node{\scriptsize$5$};
\def\s{0.2};
\draw (0-\s,0-\s) circle(\s*1.414);
\draw (2+\s,0-\s) circle(\s*1.414);
\draw (2+\s,2+\s) circle(\s*1.414);
\draw (0-\s,2+\s) circle(\s*1.414);
\def\v{0.045};
\draw [-{Stealth[length=1.6mm]}, very thin] (0-2*\s,0-2*\s) to (0-2*\s+\v,0-2*\s-\v);
\draw [-{Stealth[length=1.6mm]}, very thin] (2+2*\s,0-2*\s) to (2+2*\s+\v,0-2*\s+\v);
\draw [-{Stealth[length=1.6mm]}, very thin] (2+2*\s,2+2*\s) to (2+2*\s-\v,2+2*\s+\v);
\draw [-{Stealth[length=1.6mm]}, very thin] (0-2*\s,2+2*\s) to (0-2*\s-\v,2+2*\s-\v);
\def\w{0.15};
\draw (0-2*\s-\w,0-2*\s-\w)node{\scriptsize$2$};
\draw (2+2*\s+\w,0-2*\s-\w)node{\scriptsize$2$};
\draw (2+2*\s+\w,2+2*\s+\w)node{\scriptsize$2$};
\draw (0-2*\s-\w,2+2*\s+\w)node{\scriptsize$2$};
\end{tikzpicture}
\end{center}
Here each edge is labeled with
a small number
indicating its multiplicity.
Hence, the matrix representations $A_Y$ and $B_Y$ of
$\mathcal{A}_Y$
and
$\mathcal{BF}_Y$
with respect to the basis $\{(v,\sigma_i)\mid i=1,5,7,11\}$
are given by
$$
A_Y=
\begin{pmatrix}
2 & 7 & 5 & 11\\
7 & 2 & 11 & 5\\
5 & 11 & 2 & 7\\
11 & 5 & 7 & 2
\end{pmatrix}
\hspace{10pt}\mathrm{and}\hspace{10pt}
B_Y=
\begin{pmatrix}
-1 & -7 & -5 & -11\\
-7 & -1 & -11 & -5\\
-5 & -11 & -1 & -7\\
-11 & -5 & -7 & -1
\end{pmatrix}
$$
respectively.
By calculating the invariant factors,
we obtain
$$
\BF{Y}=\coker\mathcal{BF}_Y\cong
\Z\oplus\Z/24\Z\oplus\Z/24\Z.
$$
Hence, we have $\#\BFtors{Y}=2^6\cdot3^2$.
This is the desired result.

Next, 
we consider the case where $K=\Q(\zeta_9)$.
We set $S:=\{3,\infty\}$ and
set $G=G_m=\Gal(K/\Q)$.
In this case,
Assumption \ref{NTZ} is satisfied
by Remark \ref{rmk_NTZ_holds}.
It is known that $h_K^-=1$.
Therefore, the right hand side of Corollary~\ref{corcycMTH} is equal to
$3^6$.

We define a bouquet digraph $X$ and its derived graph $Y$
in the same way as before.
Then the matrix $B_Y$ is given by
$$
-B_Y=
\begin{pmatrix}
1&2&4&5&7&8\\
5&1&2&7&8&4\\
7&5&1&8&4&2\\
2&4&8&1&5&7\\
4&8&7&2&1&5\\
8&7&5&4&2&1
\end{pmatrix}.
$$
By calculating the invariant factors, we obtain
$$
\BF{Y}=\coker\mathcal{BF}_Y\cong
\Z\oplus\Z\oplus\Z/9\Z\oplus\Z/9\Z\oplus\Z/9\Z.
$$
Hence, we have $\#\BFtors{Y}=3^6$.
This is the desired result.
\subsection{Another imaginary abelian field}\label{__EG3}
First, 
we provide a convenient lemma
for verifying Assumption \ref{NTZ}.
While this lemma concerns cyclotomic fields,
it provides a sufficient condition
for Assumption \ref{NTZ}
in the case of imaginary abelian fields.
Indeed, every imaginary abelian field $K$ with conductor $m$
can be embedded into $\Q(\zeta_m)$,
and if Assumption \ref{NTZ} holds for $G_m$,
then it also holds for $G$.
\begin{lem}\label{eg_LEM}
Let $p$ and $q$ be distinct odd primes and 
let $a$ and $b$ be positive integers.
We set $m:=p^aq^b$.
Then Assumption \ref{NTZ} for $G_m=\Gal(\Q(\zeta_m)/\Q)$ is satisfied
if and only if
$-1\in\braket{p}\subset(\Z/q\Z)^\times$ and
$-1\in\braket{q}\subset(\Z/p\Z)^\times$.
\end{lem}
\begin{proof}
Fix an odd character $\chi\in\CG{G_m}$ with conductor $f$,
and write $\chi(\ell)$ for the value
at a rational prime $\ell$
of the primitive character associated with $\chi$.
If $p\mid f$ and $q\mid f$, then
$\chi(p)=\chi(q)=0\ne1$,
so that Assumption \ref{NTZ} holds.

We consider the case where $f=p^c\ (1\leq c\leq a)$.
First, we show that $-1\in\braket{q}\subset(\Z/p\Z)^\times$
if and only if $-1\in\braket{q}\subset(\Z/p^c\Z)^\times$.
For each $i\in\{1, c\}$,
let $k_i$ be the order of $q$ in $(\Z/p^i\Z)^\times$.
It suffices to show that
$k_1$ is even if and only if $k_c$ is even.
Since $q^{k_c}\equiv1\mod p^c$ implies $q^{k_c}\equiv1\mod p$,
it follows that $k_1\mid k_c$.
Hence $k_c$ is even if $k_1$ is even.

On the other hand,
from $q^{k_1}\equiv1\mod p$,
we have $q^{pk_1}\equiv1\mod p^2$.
By induction, we obtain
$q^{p^{c-1}k_1}\equiv1\mod p^c$.
It follows that $k_c\mid p^{c-1}k_1$.
Hence $k_1$ is even if $k_c$ is even
since $p$ is odd.

Next, we assume that $-1\in\braket{q}\subset(\Z/p^c\Z)^\times$.
If $\chi(q)=1$, then we have
$$
-1=\chi(-1)=\chi(q)^{\frac{k_c}{2}}=1,
$$
which is a contradiction.
Therefore, we have $\chi(q)\ne1$.
It is clear that $\chi(p)\ne1$.
The case where $f=q^d\ (1\leq d\leq b)$ is similar.
Hence, in this case, Assumption \ref{NTZ} holds.

Conversely, we may assume that
$-1\not\in\braket{q}\subset(\Z/p\Z)^\times$.
Hence, there exists an odd character $\overline{\psi}$ on $(\Z/p\Z)^\times/\braket{q}$.
Let $\psi$ be the character on $G_m$
obtained by composing $\overline{\psi}$ with
the natural surjection $G_m\to(\Z/p\Z)^\times/\braket{q}$.
Then $\psi^\mathrm{prim}(q)=1$.
This completes the proof.
\end{proof}

We consider the case where $K:=\Q(\sqrt{-15}, \zeta_9+\zeta_9^{-1})$.
Then we have $m=3^2\cdot5$.
We set $S=\{3, 5, \infty\}$.
Let $G$ and $G_m$ be as in \S\ref{__MTH}.
In this case, 
since
$$
3^2\equiv-1\mod 5
\hspace{10pt}
\text{and}
\hspace{10pt}
5^1\equiv-1\mod3,
$$
Assumption \ref{NTZ} is satisfied by Lemma \ref{eg_LEM}.
The subfield structure of $K$
is illustrated in the diagram below.
\begin{center}
\begin{tikzpicture}[scale=0.86]
\draw(0,0) node{$\Q(\zeta_{45})$};
\draw(0,-0.35)--(0,-1.15);
\draw(0,-1.5) node{$K=\Q(\sqrt{-15}, \zeta_9+\zeta_9^{-1})$};
\draw(-1,-1.85)--(-2,-2.65);
\draw(-2,-3) node{$\Q(\sqrt{-15})$};
\draw(1,-1.85)--(2,-2.65);
\draw(2.5,-3) node{$\Q(\zeta_9+\zeta_9^{-1})=K^+$};
\draw(-2,-3.35)--(-0.5,-4.45);
\draw(2,-3.35)--(0.5,-4.45);
\draw(0,-4.6) node{$\Q$};
\end{tikzpicture}
\end{center}
Since $K=K^+\Q(\sqrt{-15})$ and
both $3$ and $5$ are totally ramified in $\Q(\sqrt{-15})$,
we have $t=0$.
Moreover, we have
$[\Q(\zeta_{45}):K]=4$, $w_K=2$, and $\mathcal{Q}_K=1$.
It is known that $h_{K}^-=2$.
Therefore, the right hand side of Theorem \ref{MTH} is
equal to $2^3\cdot3^8\cdot5^4$.

We consider the graph side.
We define the bouquet digraph $X$ and
the voltage assignment $\alpha$ as in \S\ref{__MTH}.
In this case,
we have $G\cong\Z/3\Z\times\Z/2\Z$.
Let $\rho$ and $\tau$ be generators
of $\Z/3\Z$ and $\Z/2\Z$
respectively.
Then, the images of $\sigma_i\in G_m$ under
the natural projection $\pi:G_m\to G$
can be classified as follows:
$$
\pi(\sigma_i)=\left\{ \,
					\begin{aligned}
						& (1,1)\ &&(i=1,8,17,19)\\
						& (1,\tau)\ &&(i=26,28,37,44)\\
						& (\rho,1)\ &&(i=2,16,34,38)\\
						& (\rho,\tau)\ &&(i=7,11,29,43)\\
						& (\rho^2,1)\ &&(i=4,23,31,32)\\
						& (\rho^2,\tau)\ &&(i=13,14,22,41).
					\end{aligned}
				\right.
$$
The derived digraph $Y:=X(G,\pi\circ\alpha)$ is given by the following graph.
\begin{center}
\begin{tikzpicture}
\fill[white] (0,1) circle(0);
\fill[white] (0,-1) circle(0);
\def\r{2};
\def\t{60};
\coordinate (O) at (0,0);
\coordinate (A0) at ({\r*cos(0*60)},{\r*sin(0*60)});
\coordinate (A1) at ({\r*cos(1*60)},{\r*sin(1*60)});
\coordinate (A2) at ({\r*cos(2*60)},{\r*sin(2*60)});
\coordinate (A3) at ({\r*cos(3*60)},{\r*sin(3*60)});
\coordinate (A4) at ({\r*cos(4*60)},{\r*sin(4*60)});
\coordinate (A5) at ({\r*cos(5*60)},{\r*sin(5*60)});
\fill[black] (A0) circle (0.08);
\fill[black] (A1) circle (0.08);
\fill[black] (A2) circle (0.08);
\fill[black] (A3) circle (0.08);
\fill[black] (A4) circle (0.08);
\fill[black] (A5) circle (0.08);
\draw[thin](A0)--(A1);
\draw[thin](A0)--(A2);
\draw[thin](A0)--(A3);
\draw[thin](A0)--(A4);
\draw[thin](A0)--(A5);
\draw[thin](A1)--(A2);
\draw[thin](A1)--(A3);
\draw[thin](A1)--(A4);
\draw[thin](A1)--(A5);
\draw[thin](A2)--(A3);
\draw[thin](A2)--(A4);
\draw[thin](A2)--(A5);
\draw[thin](A3)--(A4);
\draw[thin](A3)--(A5);
\draw[thin](A4)--(A5);
\def\s{0.95};
\def\ss{1.00};
\def\u{1.2};
\def\v{1.393};
\def\vv{1};
\def\vvv{2}
\def\w{1.1};
\draw ({\s*\r*cos(1.5*60)},{\s*\r*sin(1.5*60)})node{\scriptsize{$135$}};
\draw ({\ss*\r*cos(3.5*60)},{\s*\r*sin(3.5*60)})node{\scriptsize{$135$}};
\draw ({\ss*\r*cos(5.5*60)},{\s*\r*sin(5.5*60)})node{\scriptsize{$135$}};
\draw ({\u*\r*cos(0*60)},{\u*\r*sin(0*60)}) circle(0.4);
\draw ({\u*\r*cos(1*60)},{\u*\r*sin(1*60)}) circle(0.4);
\draw ({\u*\r*cos(2*60)},{\u*\r*sin(2*60)}) circle(0.4);
\draw ({\u*\r*cos(3*60)},{\u*\r*sin(3*60)}) circle(0.4);
\draw ({\u*\r*cos(4*60)},{\u*\r*sin(4*60)}) circle(0.4);
\draw ({\u*\r*cos(5*60)},{\u*\r*sin(5*60)}) circle(0.4);
\draw [-{Stealth[length=1.6mm]}, very thin] ({\v*\r*cos(0*60-\vv+\vvv)},{\v*\r*sin(0*60-\vv+\vvv)}) to ({\v*\r*cos(0*60+\vv+\vvv)},{\v*\r*sin(0*60+\vv+\vvv)});
\draw [-{Stealth[length=1.6mm]}, very thin] ({\v*\r*cos(1*60-\vv+\vvv)},{\v*\r*sin(1*60-\vv+\vvv)}) to ({\v*\r*cos(1*60+\vv+\vvv)},{\v*\r*sin(1*60+\vv+\vvv)});
\draw [-{Stealth[length=1.6mm]}, very thin] ({\v*\r*cos(2*60-\vv+\vvv)},{\v*\r*sin(2*60-\vv+\vvv)}) to ({\v*\r*cos(2*60+\vv+\vvv)},{\v*\r*sin(2*60+\vv+\vvv)});
\draw [-{Stealth[length=1.6mm]}, very thin] ({\v*\r*cos(3*60-\vv+\vvv)},{\v*\r*sin(3*60-\vv+\vvv)}) to ({\v*\r*cos(3*60+\vv+\vvv)},{\v*\r*sin(3*60+\vv+\vvv)});
\draw [-{Stealth[length=1.6mm]}, very thin] ({\v*\r*cos(4*60-\vv+\vvv)},{\v*\r*sin(4*60-\vv+\vvv)}) to ({\v*\r*cos(4*60+\vv+\vvv)},{\v*\r*sin(4*60+\vv+\vvv)});
\draw [-{Stealth[length=1.6mm]}, very thin] ({\v*\r*cos(5*60-\vv+\vvv)},{\v*\r*sin(5*60-\vv+\vvv)}) to ({\v*\r*cos(5*60+\vv+\vvv)},{\v*\r*sin(5*60+\vv+\vvv)});
\draw({\w*\v*\r*cos(0*60)},{\w*\v*\r*sin(0*60)})node{\scriptsize{$46$}};
\draw({\w*\v*\r*cos(1*60)},{\w*\v*\r*sin(1*60)})node{\scriptsize{$46$}};
\draw({\w*\v*\r*cos(2*60)},{\w*\v*\r*sin(2*60)})node{\scriptsize{$46$}};
\draw({\w*\v*\r*cos(3*60)},{\w*\v*\r*sin(3*60)})node{\scriptsize{$46$}};
\draw({\w*\v*\r*cos(4*60)},{\w*\v*\r*sin(4*60)})node{\scriptsize{$46$}};
\draw({\w*\v*\r*cos(5*60)},{\w*\v*\r*sin(5*60)})node{\scriptsize{$46$}};
\end{tikzpicture}
\end{center}
Here, for simplicity, each undirected edge is a multi-edge,
meaning that
a pair of directed edges in opposite directions is
represented by a single edge,
and the multiplicity of unlabeled edges is $90$.
Hence, the matrix representation $B_Y$ of
the Bowen--Franks operator 
$\mathcal{BF}_Y$
with respect to the basis $\{(v,(\rho^i,\tau^j))\mid i=0,1,2,\ j=0,1\}$
is given by
$$
B_Y=
\begin{pmatrix}
-45 & -135 &&&&\\
-135 & -45 &&&&\\
&&-45 & -135 &&\\
&&-135 & -45 &&\\
&&&&-45 & -135\\
&&&&-135 & -45
\end{pmatrix}
$$
where all blank entries are equal to $-90$.
By calculating the invariant factors,
we obtain
$$
\BF{Y}=\coker\mathcal{BF}_Y\cong
\Z^2\oplus(\Z/45\Z)^3\oplus\Z/360\Z.
$$
Hence, we have $\#\BFtors{Y}=2^3\cdot3^8\cdot5^4$.
This is the desired result.

\end{document}